\documentclass[a4paper,12pt]{amsart}
\usepackage{fullpage}
\usepackage{amssymb}
\usepackage{amsthm}
\usepackage{mathtools}
\usepackage{graphicx}
\usepackage{tikz}
\usepackage{tikz-cd}

\usepackage[titletoc]{appendix}

\newcounter{counter2}
\numberwithin{equation}{section}
\numberwithin{counter2}{section}

\newtheorem{thm}[counter2]{Theorem}
\newtheorem{prop}[counter2]{Proposition}
\newtheorem{lemma}[counter2]{Lemma}

\newtheorem{defn}[counter2]{Definition}
\newtheorem{rem}[counter2]{Remark}

\newtheorem{manualtheoreminner}{Theorem}
\newenvironment{thmr}[1]{%
  \renewcommand\themanualtheoreminner{#1}%
  \manualtheoreminner
}{\endmanualtheoreminner}

\newcommand{\D}{\partial}

\DeclareSymbolFont{script}{U}{eus}{m}{n}
\DeclareMathSymbol{\Wedge}{0}{script}{"5E}

\newcounter{mnotecount}[section]

\renewcommand{\themnotecount}{\thesection.\arabic{mnotecount}}

\newcommand{\mnote}[1]
{\protect{\stepcounter{mnotecount}}$^{\mbox{\footnotesize
$
\bullet$\themnotecount}}$ \marginpar{
\raggedright\tiny\em
$\!\!\!\!\!\!\,\bullet$\themnotecount: #1} }

\newcommand{\CP}{\mathbb{CP}}                 

\newcommand{\C}{\mathbb{C}}

\newcommand{\PP}{\mathbb{P}}

\newcommand{\X}{{X}}
\newcommand{\Y}{\mathcal{Z}}
\def\p{\partial}
\def\be{\begin{equation}}
\def\D{\mathcal{D}}
\def\dim{\mbox{dim}}

\def\dim{\mbox{dim}}

\def\ee{\end{equation}}

\def\bea{\begin{eqnarray}}
\def\eea{\end{eqnarray}}
\newcommand{\spp}{\mathbb{S}}

\begin{document}
\title{Heavenly metrics,  hyper-Lagrangians and Joyce structures}

\author{Maciej Dunajski}
\address{Department of Applied Mathematics and Theoretical Physics\\ 
University of Cambridge\\ Wilberforce Road, Cambridge CB3 0WA, UK.}
\email{m.dunajski@damtp.cam.ac.uk}

\author{Timothy Moy}
\address{Department of Applied Mathematics and Theoretical Physics\\ 
University of Cambridge\\ Wilberforce Road, Cambridge CB3 0WA, UK.}
\email{tjahm2@cam.ac.uk} 
\date{\today}

\begin{abstract}
In \cite{B3},  Bridgeland defined a geometric structure,  named a Joyce structure,  conjectured to exist on the space $M$ of stability conditions of a $CY_3$ triangulated category.  Given a non-degeneracy assumption,  a feature of this structure is a complex hyper-K\"ahler metric with homothetic symmetry on the total space $X = TM$ of the holomorphic tangent bundle.  \par
Generalising the isomonodromy calculation which leads to the $A_2$ Joyce structure in \cite{BM},  we obtain an explicit expression for a hyper-K\"ahler metric with homothetic symmetry via construction of the isomonodromic flows of a Schr\"odinger equation with deformed polynomial oscillator potential of odd degree $2n+1$.  The metric is defined on a total space $X$ of complex dimension $4n$ and fibres over a $2n$--dimensional manifold $M$ which can be identified with the unfolding of the $A_{2n}$-singularity.  The hyper-K\"ahler structure is shown to be compatible with the natural symplectic structure on $M$ in the sense of admitting an \textit{affine symplectic fibration} as defined in \cite{BS}.   \par 
Separately,  using the additional conditions imposed by a Joyce structure,  we consider reductions of Pleba\'nski's heavenly equations that govern the hyper-K\"ahler condition.  We introduce the notion of a \textit{projectable hyper-Lagrangian} foliation and
show that in dimension four such a foliation of $X$ leads to a linearisation of the heavenly equation.  The hyper-K\"ahler metrics constructed here are shown to admit such a foliation.  
\end{abstract}
\maketitle
\section{Introduction}

In \cite{B3} it was proposed that the Donaldson-Thomas invariants
of a three--dimensional Calabi--Yau triangulated ($CY_3$) category can be encoded in a geometric
structure on its space $M$ of stability conditions.  This structure,  named a \textit{Joyce structure} may be outlined as follows: The complex manifold $M$ carries a closed holomorphic $2$-form taking rational values on a system of lattices in the holomorphic tangent bundle $X = TM$.  Assuming the $2$-form is symplectic (which we will do in the sequel) there is a compatible complex
hyper-K\"ahler structure on the total space $X$.  In \cite{BS} it was shown how this hyper-K\"ahler structure is encoded in a solution
to an over-determined system of nonlinear PDEs for a single function $\Theta$.  Additional homogeneity and lattice invariance conditions
must be imposed on $\Theta$ so that it defines a Joyce structure.  \par In the lowest--dimensional case of interest, where $X$ is a complex four--manifold, the system reduces
to a single PDE which is Pleba\'nski's second heavenly equation for anti--self--dual, Ricci--flat
complex metrics \cite{P2}.   
In this case the homogeneity condition is a conformal symmetry reduction in a class 
studied in \cite{DT}.

In \cite{BS2} Bridgeland and Smith realise spaces of stability conditions of $CY_3$ triangulated categories associated to certain quivers as spaces of quadratic differentials on Riemann surfaces.  This motivates constructions of Joyce structures starting from the data of a space of quadratic differentials.  The base of the Joyce structure should be $M = \text{Quad}(\gamma,\mu)$, the moduli space of Riemann surfaces of genus $\gamma$ equipped with meromorphic quadratic differential with $l$ poles of orders $\mu = \{m_1,...,m_l\}$ and simple zeroes.   
There are such constructions for holomorphic quadratic differentials ($\mu = \emptyset$) in \cite{B2} and in the genus zero case,  for $\mu = \{7\}$ in \cite{BM}.  \par 
The latter construction of \cite{BM} proceeds by calculating the isomonodromic flows for a family of second-order linear ODE.  The ODE can be thought of as deformations of Schr\"odinger's equation with a cubic oscillator potential and are specified by points in the manifold $X = TM$.  
Interpreting the flows as a \textit{Lax pair}  on $X\times\CP^1$  gives rise to a complex hyper-K\"ahler metric on $X$.  \par 
A primary aim of this article is to generalise this construction.  
We construct the isomonodromic flows  corresponding to  a family of ODEs with deformed polynomial potential of general odd degree $2n+1$.  This allows us to explicitly construct a complex hyper-K\"ahler metric with homothety on a complex manifold $X$.  Each point in $X$ specifies an ODE in the family,  and $X$ fibres over $M = \text{Quad}(0,\{2n+5\})$,  the moduli space of quadratic differentials on $\mathbb{CP}^1$ with a pole of order $2n+5$ and simple zeroes.   
The hyper-K\"ahler metric is compatible with the symplectic form $\omega$ on $M$ induced  \cite{B2}  by the cohomology intersection forms on $H^1(\Sigma_p,  \mathbb{C})$ where we associate to each point $p \in M$ a hyper-elliptic curve $\Sigma_p$.  If $n=1$, then $M$ identified with the  universal unfolding of the $A_{2}$ singularity carries a Frobenus structure
in the sense of Dubrovin \cite{D3} which can be constructed directly from the complex hyper--K\"ahler structure.
\par \par 

The structure of the article is as follows: \S \ref{review} is a review of complex hyper-K\"ahler structures:
$4n$--complex dimensional manifolds $X$ with a holomorphic Riemannian metric $g$, and a triple
of holomorphic, Hermitian, endomorphisms $I, J, K$ of $TX$ satisfying the quaternion relations
and parallel for the Levi--Civita connection of $g$.
We do this,  taking the spinorial approach of Bailey and Eastwood to \textit{paraconformal structures}, introduced in \cite{BE}.  We formulate the hyper-K\"ahler condition as two equivalent systems of PDEs of a single function which generalise Pleba\'nski's first and second heavenly equations \cite{P2}.  We show that the second system can be interpreted as an infinitesimal version of the first.  We explain how a \textit{Lax distribution} on $X \times \mathbb{CP}^1$ defines a paraconformal structure with an associated family of metrics and give sufficient conditions that this family contains a hyper-K\"ahler metric.  

In \S \ref{hyper--lagrangian} - \S \ref{reductions} we consider complex hyper-K\"ahler metrics with extra symmetry.  The discussion is independent of the construction of the complex hyper-K\"ahler metrics presented in \S \ref{A2andD21} -  \S \ref{A_N metric}.
 In \S \ref{hyper--lagrangian} we introduce the notion of a \textit{projectable hyper-Lagrangian foliation}.  
\begin{defn}[Projectable hyper-Lagrangian foliation]
Let $(X, g)$ be a complex hyper--K\"ahler structure of dimension $4n$ for complex structures $I, J, K$.    An integrable rank-$2n$ distribution $\mathcal{B} \subseteq TX$ is called a hyper--Lagrangian foliation if $\mathcal{B}$ is Lagrangian with respect to the holomorphic symplectic forms associated to endomorphisms $I, J, K$. 

The hyper--Lagrangian foliation is called projectable if,  writing $\pi:X\rightarrow M$,  for the natural fibration over $M = X / \operatorname{ker} N$ where $N = (J-iK)/2$ there exists a rank--$n$ foliation $\mathcal{L}$ of $M$ with  $d\pi(\mathcal{B}) = \mathcal{L}$.
\end{defn} 
We show such foliations have a characterisation in terms of the existence of adapted local coordinates as follows:  \begin{thm}
\label{theo1_intro}
The following conditions are equivalent for a complex hyper--K\"ahler manifold $(\X, g)$ 
\begin{enumerate}
\item $\X$ is foliated by projectable hyper--Lagrangian submanifolds.
\item At at each point of $M$ there exists a local coordinate system $(x^i,y^i)$,
such that the second Pleba\'nski potential corresponding to the metric $g$ is at most quadratic in half of the 
fibre coordinates in the fibration $\pi: X\rightarrow M$.
\end{enumerate}
If $n=1$ then the conditions \rm(1)\it and \rm(2)\it \! are equivalent to the existence of a two--parameter family
of anti--self--dual null surfaces in $X$ which push down to curves in $M$.  
\end{thm} 
 We recall the definition and relevant properties of Joyce structures in \S \ref{joyce}.
 In \S \ref{reductions} we present a number of new reductions and solutions of the heavenly equations in four dimensions with relevance to Joyce structures.   We consider the case of a projectable hyper-Lagrangian foliation.  We obtain:
 \begin{thm}\label{theo2_intro}
 Suppose a complex hyper--K\"ahler manifold $(\X, g)$ of complex dimension four is foliated  by projectable hyper--Lagrangian submanifolds. Then in the adapted coordinates of Theorem \rm{\ref{theo1_intro}} \it the heavenly equation for the Pleba\'nski potential reduces to a system which is linearisable by a contact transformation.  
 \end{thm}
We also consider the stronger condition of the foliation pushing down to all twistor-fibres.  In this case we can solve the first heavenly equation and give the general form of the metric in terms of two unconstrained holomorphic functions.  We consider the case where the vector field generating the homothetic symmetry of a Joyce structure is null,  in which case we show the metric coincides with a generalisation of the $\mathcal{H}$-space of Sparling and Tod \cite{ST}.  \par 
 In \S \ref{A2andD21} we review the isomonodromy problem for the deformed cubic oscillator and explain how it is related to one studied by Okamoto \cite{O} and is associated with the Painlev\'{e} I equation.  We use the Lax formulation of the hyper--K\"ahler condition to construct the explicit form of the hyper--K\"ahler metric in this case.
In \S \ref{isomonodromy} we shall consider a deformed  polynomial quantum oscillator governed by the
Schr\"odinger equation
\begin{equation}\label{schro0}
\hbar^2 \frac{d^2y}{dx^2} = Q(x)y, \quad\text{where}\quad Q(x) :=Q_0(x) + \hbar {Q_1(x)} + \hbar^2 Q_2(x).
\end{equation}
Here
\[
Q_0(x)=x^{2n+1}+c_{2n-1}x^{2n-1}+\dots+c_0
\]
corresponds to a quadratic differential $Q_0(x)dx^2$ on $\CP^1$ with a pole of order $2n+5$ at $\infty$ paramerised by a point with coordinates $c_0, \dots, c_{2n-1}$ in the $2n$--dimensional symplectic manifold
$M=\text{Quad}(0, 2n+5) $. The first deformation term is
\[
Q_1(x) = \sum_{i=1}^n \frac{p_i}{x-q_i}+ R(x) 
\]
where $R(x)$ is the general polynomial of degree at most $n-1$.  Finally the trailing term $Q_2(x)$ is specified so that the ODE (\ref{schro0}) has `square-root-like' monodromy at the regular singularity $q_i$. 

We shall compute the $2n$ linearly-independent isomonodromic flows of (\ref{schro0}), and show that  dependence of the isomonodromic flows on a \textit{spectral parameter} $\lambda=\hbar$ defines a $(2n, 2)$-paraconformal structure on the $4n$--dimensional manifold $X$ of potentials $Q(x)$ in (\ref{schro0}). To every such paraconformal structure is associated a family of metrics.  
In \S \ref{A_N metric} we will distinguish a complex hyper-K\"ahler metric in this family by 
insisting that it is compatible with the natural symplectic structure on $M$.  
\begin{thm}
Let $M=\mathrm{Quad}(0, {2n+5})$ be the $2n$--complex--dimensional manifold of quadratic differentials
differentials on $\CP^1$ with a single pole of order $2n+5$ and simple zeros,  equipped with its canonical symplectic form $\omega$, and
let $X$ be the $4n$--complex--dimensional holomorphic manifold parametrising the choices of $Q_0(x)$ and $Q_1(x)$ in 
the Schr\"odinger equation (\ref{schro0}). The family of metrics associated to the paraconformal structure on $X$ defined by
the isomonodromic flows of (\ref{schro0}) contains a
  complex hyper--K\"ahler metric $g$
such that $\pi^*\omega=g((J+iK)/2, \cdot)$ for some compatible complex structures $J, K$ on $M$.
This metric admits a homothetic Killing vector field, and a projectable hyper--Lagrangian foliation.  
\label{theo3_intro}
\end{thm}
\par 
\subsection*{Acknowledgements}
The authors would like to thank Tom Bridgeland for very helpful suggestions and also the anonymous referees for their detailed comments which have lead to improvements
in the manuscript.  Timothy Moy is supported by Cambridge Australia Scholarships.  
\section{Paraconformal structures and complex hyper-K\"ahler metrics}\label{review}
\subsection{Paraconformal structures}
\label{section2}
Let $X$ be a complex manifold of complex dimension $4n$.  Denote by $TX$ the holomorphic tangent bundle.  A complex hyper-K\"ahler structure is a triple of holomorphic endomorphisms $I, J, K$ of $TX$ satisfying the
quaternion relation
\[
I^2=J^2=K^2=IJK= -\text{Id}_{TX}
\]
which are Hermitian for a holomorphic metric $g$ on $X$ and parallel for its Levi-Civita connection.  \par 
Complex hyper--K\"ahler metrics fit into the framework of paraconformal structures \cite{BE} which we now define 
\begin{defn}[Paraconformal structure]
A $(2n, 2)$-paraconformal structure on $X$ is an isomorphism 
\begin{align}
\label{paraconeq}
TX \cong \spp\otimes \spp'
\end{align}
where $ \spp$ and $\spp'$ are holomorphic vector bundles of rank $2n$ and rank $2$ respectively.
\end{defn}  Associated to a paraconformal structure is the family of holomorphic \textit{compatible metrics} of the form 
\begin{align}\label{compatible_metric}
g = e \otimes \epsilon 
\end{align}
where $ e,  \epsilon$ are non-degenerate skew-forms on $ \spp$ and $\spp'$ respectively.  Since $\epsilon$ is determined up to scale,  we may fix this scale and a choice of compatible metric becomes a choice of non-degenerate skew-form $e$ on $\spp$.  \par
\begin{defn}[Null structure] A null structure on an $4n$-complex dimensional manifold $X$ with holomorphic metric $g$ is an endomorphism $N:TX\rightarrow TX$ such that \[
N^2=0, \quad g(N\cdot,\cdot) = -g(\cdot,N\cdot), \quad \text{rank(N)}=2n.
\]
The null structure is called null--K\"ahler \cite{D} if $N$ is parallel for the Levi--Civita connection of $g$.
\label{defnN}
\end{defn}
A section $o \in \Gamma(\spp')$ defines a null structure by
\begin{align}
\label{onull}
N := \text{Id}_{\spp} \otimes \epsilon(o,\cdot)\otimes o.  
\end{align} 
Given an additional section $\iota \in \Gamma(\spp')$ for which $\epsilon(o,\iota) = 1$,  we may define endomorphisms of $TX$
\begin{align}\label{quaternions}
I &=  -\text{Id}_{\spp} \otimes i\epsilon(o,\cdot)\otimes \iota \,   - \text{Id}_{\spp} \otimes  i\epsilon(\iota,\cdot)\otimes o  \\ 
J &=  \ \text{Id}_{\spp} \otimes  \ \epsilon(o,\cdot)\otimes o   + \text{Id}_{\spp} \otimes  \ \epsilon(\iota,\cdot)\otimes \,  \iota \nonumber \\ 
K &= \ \text{Id}_{\spp} \otimes  i\epsilon(o,\cdot)\otimes o \, - \text{Id}_{\spp} \otimes i\epsilon(\iota,\cdot)\otimes \iota \ \nonumber
\end{align}
which are by construction Hermitian for $g$.  We define an associated non-degenerate $2$-form $\Omega_I := g(I\cdot,\cdot)$ with $\Omega_J$ and $\Omega_K$ defined similarly.  Note that 
\begin{equation*}
N = \frac{1}{2}(J-iK).  
\end{equation*} \par 
Suppose that there exists connections on $ \spp$ and $\spp'$ which preserve $ e,  \epsilon$ and for which the induced connection on $TX$ is the Levi-Civita connection for $g$.  Suppose also that the section $o$ is parallel.  Then $N$ together with $g$ is the general null-K\"ahler structure as defined in \cite{D}.   
Next,  if $\iota$ is also parallel,  so the connection on $\spp'$ is flat,  then $I,J,K$ together with $g$ define a complex hyper-K\"ahler structure.  Every complex hyper-K\"ahler structure (locally) arises this way \cite{BS, D}.  
\subsection{Notation}
Unless otherwise specified we will employ the Einstein summation convention.  Unprimed lower case indices will take values in $1,...,2n$ whereas primed indices take values in $0,1$.  We define skew-symmetric matrices with components $\eta_{ij}$ and $\epsilon_{i'j'}$ by
\begin{equation}\label{symplectic_convention}
\begin{aligned}
\eta_{ij} &= 0, \quad  i, j \le n  \\
\eta_{ij} &= 0, \quad  i, j > n \nonumber \\
\eta_{i(j+n)} &= \delta_{ij}, \, \ i,j \le n ,  \nonumber  
\end{aligned}
\ \ \ \ \ 
\begin{aligned}
\text{equivalently } \ \ \ \eta = 
\begin{bmatrix}
0&\operatorname{Id}_n\\
-\operatorname{Id}_n&0
\end{bmatrix}
\end{aligned}
\end{equation}
and 
\begin{equation*}
\epsilon_{0'1'} = 1.
\end{equation*}
We use the convention of \cite{PR1},  in that the components $\eta^{ij},  \epsilon^{i'j'}$ of the inverses are defined by the relations
\begin{equation}
\eta^{ik}\eta_{jk} = \delta^{i}_{j},  \quad \epsilon^{i'k'}\epsilon_{j'k'} = \delta^{i'}_{j'}.
\end{equation}
We define the symmetric tensor product according to: $a \odot b\equiv\frac{1}{2}(a\otimes b+b\otimes a)$.  
\subsection{Pleba\'nski's heavenly equations}
We now explain how to encode the complex hyper-K\"ahler condition in two equivalent system of PDEs.  
The following proposition 
goes back to Pleba\'nski \cite{P2} if $\text{dim}(X)=4$, and \cite{BS, D,AMS} if $\dim{(X)}>4$.  We give a self-contained proof of it,  which clarifies some of the issues with a freedom to modify the scalar potentials
and reabsorb redundant functions if $\text{dim}(X)>4$.
\begin{prop}
Let $(X, g)$ be a complex hyper--K\"ahler manifold of dimension $4n$. There exist local
coordinate systems $(z^j, \tilde{z}^j)$ and $(x^j, y^j)$, and functions $U=U(z, \tilde{z})$,  the first Pleba\'nski potential,  and 
$\Theta=\Theta(x, y)$,  the second Pleba\'nski potential,  such that the metric $g$ takes the form
\be
g=\frac{\p U}{\p z^j\p\tilde{z}^k} dz^j \odot d\tilde{z}^k, \quad\text{where}
\label{pleb1_form}
\ee
\begin{align}\label{heavenly1}
  \frac{\partial^2 U}{\partial z^k \partial \tilde{z}^i} \frac{\partial^2 U}{\partial z^l \partial \tilde{z}^j}\eta^{kl} = \eta_{ij},  
\end{align}
and
\begin{align}\label{pleb2_form}
g = \eta_{ij}dy^i \odot dx^j + \frac{\partial^2 \Theta}{\partial y^i \partial y^j} dx^i \odot dx^j, \quad
\text{where} 
\end{align}
\begin{align}\label{heavenly2}
 \frac{\partial^2 \Theta}{\partial y^i \partial x^j} - \frac{\partial^2 \Theta}{\partial y^j \partial x^i}  - \eta^{kl} \frac{\partial^2 \Theta}{\partial y^i \partial y^k} \frac{\partial^2 \Theta}{\partial y^j \partial y^l} = 0.
\end{align}
\end{prop}
\begin{proof}
Define the following $2$-forms 
\begin{align}\label{null_2-forms}
\Omega_{+} &:= e \otimes \epsilon(\iota,\cdot) \otimes \epsilon(\iota,\cdot) \\
\Omega_{-} &:= e \otimes \epsilon(o,\cdot) \otimes \epsilon(o,\cdot) = g(N\cdot,\cdot).  
\end{align}
which are closed in the hyper-K\"ahler case.   The (integrable) kernels of these $2$-forms are the $\mp i$-eigenspaces of $I$ respectively and the closure implies they descend to symplectic forms on the space of leaves of their respective kernels.  The $i$-eigenspace coincides with the kernel of $N$.  
Darboux's theorem implies there exists coordinates $(z^i,  \tilde{z}^i)$,  $i = 1, ..., 2n$ for $X$ such 
that 
 \begin{align*}
   \Omega_{+} &= \frac{1}{2}\eta_{ij}d\tilde{z}^i \wedge d\tilde{z}^j \\
 \Omega_{-} &= \frac{1}{2}\eta_{ij}dz^i \wedge dz^j
 \end{align*} 
where $\eta_{ij}$ is the standard symplectic matrix given by (\ref{symplectic_convention}).  \par We also have the symplectic form
 \begin{align*}
 \Omega_{I} &= g(I\cdot,\cdot) = -2ie \otimes \epsilon(o,\cdot) \odot \epsilon(\iota,\cdot) = -\frac{i}{2}U_{ij}dz^i \wedge d\tilde{z}^j 
 \end{align*}
 where the $U_{ij}$ are some holomorphic functions to be determined.  
 $ \Omega_{I}$ is locally exact and so we may write
 \begin{align*}
 U_{ij}dz^i \wedge d\tilde{z}^j  = d(-\phi_i dz^i + \psi_i d\tilde{z}^i)
 \end{align*}
 and it follows 
 \begin{align*}
 \phi_i &= \frac{\partial \phi}{\partial z^i}, \hspace{3mm}
  \psi_i = \frac{\partial \psi}{\partial \tilde{z}^i},  \hspace{3mm}
 U_{ij} = \frac{\partial^2 U}{\partial z^i \partial \tilde{z}^j}
 \end{align*}
 where $\phi,  \psi$ are holomorphic functions on $X$ and $U = \phi + \psi$.  The metric is now given by
 (\ref{pleb1_form}).
 Similarly we may define symplectic forms $\Omega_J := g(J\cdot,\cdot)$,  $\Omega_K := g(K\cdot,\cdot)$.  In fact 
 \begin{align*}
 \Omega_J = \Omega_+ \, +\, \Omega_- =  \frac{1}{2}\eta_{ij}(dz^i \wedge dz^j + d\tilde{z}^i \wedge d\tilde{z}^j).
 \end{align*}
Then $J^{2} = -\text{Id}_{TX}$ is to say  that (\ref{heavenly1}) holds.
This is a system of PDEs generalising Pleba\'nski's first heavenly equation \cite{P2} for Ricci-flat anti-self-dual complex spacetimes in complex dimension four.  In the same article Pleba\'nski obtained another equivalent non-linear PDE of a single function which can also be generalised to a system.  To obtain it,  note that (\ref{heavenly1}) implies that $(x^i,y^i)$ give local coordinates on $X$,  where $x^i := z^i$ and $y^i$ is defined by
\begin{align*}
\eta_{ji} y^j = \frac{\partial U}{\partial z^i}.  
\end{align*}
We have
\begin{align}\label{dy}
\eta_{ji} dy^j = \frac{\partial^2U}{\partial z^i \partial z^k} dz^k +  \frac{\partial^2U}{\partial z^i \partial \tilde{z}^k} d\tilde{z}^k.
\end{align}
Define 
\begin{align*}
{\Theta}_{ij} := -\frac{\partial U}{\partial z^i \partial z^j}
\end{align*}
so that from (\ref{dy}) we see the metric is
\begin{align*}
g = \eta_{ij}dy^i \odot dx^j + {\Theta}_{ij}dx^i \odot dx^j.
\end{align*}
Also from (\ref{dy}) one may calculate 
\begin{align*}
\Omega_{+} = \frac{1}{2}\eta_{ij}dy^i \wedge dy^j + {\Theta}_{ij}dx^i \wedge dy^j + \frac{1}{2}\eta^{kl}\Theta_{ik}{\Theta}_{jl}dx^i \wedge dx^j.
\end{align*}
From the closure of $\Omega_{+}$ we obtain 
\begin{align*}
\frac{\partial {\Theta}_{ij}}{\partial y^k} - \frac{\partial {\Theta}_{ik}}{\partial y^j} = 0 
\end{align*}
which combined with the symmetry of ${\Theta}_{ij}$ implies there exists a holomorphic function $\tilde{\Theta}$ such that 
\begin{align*}
\frac{\partial^2 \tilde{\Theta}}{\partial y^i \partial y^j} = {\Theta}_{ij}.  
\end{align*}
Also from the closure
\begin{align}
\frac{\partial}{\partial y^m} \bigg(\frac{\partial^2 \tilde{\Theta}}{\partial y^i \partial x^j} - \frac{\partial^2 \tilde{\Theta}}{\partial y^j \partial x^i}  - \eta^{kl} \frac{\partial^2 \tilde{\Theta}}{\partial y^i \partial y^k} \frac{\partial^2 \tilde{\Theta}}{\partial y^j \partial y^l}  \bigg) &= 0 \label{diff_heavenly} \\ 
\frac{\partial}{\partial x^m}\bigg( \eta^{kl} \frac{\partial^2 \tilde{\Theta}}{\partial y^i \partial y^k} \frac{\partial^2 \tilde{\Theta}}{\partial y^j \partial y^l} \bigg) \ dx^m \wedge dx^i \wedge dx^j = 0. \label{closure_heavenly} 
\end{align} 
These two equations imply that
\begin{align*}
\frac{\partial^2 \tilde{\Theta}}{\partial y^i \partial x^j} - \frac{\partial^2 \tilde{\Theta}}{\partial y^j \partial x^i}  - \eta^{kl} \frac{\partial^2 \tilde{\Theta}}{\partial y^i \partial y^k} \frac{\partial^2 \tilde{\Theta}}{\partial y^j \partial y^l} = \frac{\partial \rho_i}{\partial x^j} - \frac{\partial \rho_j}{\partial x^i}  
\end{align*}
for some holomorphic functions $\rho_i(x)$ so that,  defining $\Theta := \tilde{\Theta} - y^i\rho_i(x)$ we deduce (\ref{heavenly2}), together with the metric (\ref{pleb2_form}).
\end{proof}
The system (\ref{heavenly2}) has been studied in,  for example \cite{T,  BS, D} and generalises
Pleba\'nski's second heavenly equation.  An application of the Cartan--K\"ahler theorem shows that the general solution to (\ref{heavenly2}) depends on two arbitrary functions of $(2n+1)$ variables.
There is also some coordinate freedom in the definition of $\Theta$, but this is measured by
functions of $2n$ variables.  \par 
The null-K\"ahler structure $N$ takes the form 
\begin{align}
N = dx^i \otimes {\frac{\partial}{\partial y^i}}
\end{align}
and so the kernel of $N$ is Frobenius integrable.  We may consider the leaf space
\begin{equation}
M := X / \operatorname{ker} N
\end{equation}
and associated submersion $\pi: X  \to M$.  The coordinates $x^i$ descend to coordinates on $M$ and $\Omega_{-}$ descends to a symplectic form we will label $\omega$.  Locally we may identify $X \cong TM$ by decreeing the $y^i$ to be the fibre (tangent) coordinates for $TM$ associated to the $x^i$.  This has the following convenience: coordinate transformations preserving the form of the metric (\ref{pleb2_form}) with the second Pleba\'nski potential $\Theta$ satisfying (\ref{heavenly2}) are symplectic coordinate changes on $M$ lifted to fibre coordinates on $TM$ the usual way,  together with $\Gamma(TM)$ acting by translations of the fibre coordinates.  Given such a change,  the second Pleba\'nski potential $\Theta$ needs to be modified by a function cubic in the fibre coordinates.  See \cite{BM, D} for a calculation of the modification.  
\subsection{The second heavenly system as an infinitesimal limit}
In the  thesis of the first author (\cite{D_thesis} page 38) it is  claimed that, with $n=1$,  
the second Pleba\'nski equation (\ref{heavenly2}) is the infinitesimal version of the first equation (\ref{heavenly1}). The argument below makes this precise, 
and works for the general dimension $\text{dim}(X)=4n$. 

Consider a one--parameter family of flat metrics, and corresponding volume forms 
\[
g_{\xi}=\xi^{-1}\; \eta_{ij}dz^i\odot d\tilde{z}^j, 
\]
The function $U_{\xi}=\xi^{-1} \eta_{ij}z^i\tilde{z}^j$ satisfies the  first Pleba\'nski system (\ref{heavenly1}) with the right-hand-side replaced by $\xi^{-2}\eta_{ij}$.
 Any solution to such a rescaled equation still corresponds to a complex hyper--K\"ahler metric, as
 scalling a metric by a non--zero constant does not change its hyper--K\"ahler property. 
  Now set
\[
U:=U_{\xi}-\xi^2 \Theta(x^i, y^i), \quad\mbox{where}\quad  x^i = z^i,  \ y^i=\frac{z^i-\tilde{z}^i}{\xi}.
\]
We find
\[
\frac{\p^2 U}{\p z^i \p \tilde{z}^j}=\frac{1}{\xi}\eta_{ij}+\frac{\p ^2 \Theta}{\p y^i\p y^j}+ \xi\frac{\p^2 \Theta}{\p z^i\p y^j}, \quad dz^i\odot d\tilde{z}^j=dz^i\odot dz^j+\xi dz^i\odot d y^j
\]
and
\[
g=\frac{\p^2 U}{\p z^i\p\tilde{z}^j}dz^i\odot d\tilde{z}^j= \eta_{ij}dz^i\odot dy^j
+\frac{\p^2 \Theta}{\p y^i\p y^j}dz^i\odot dz^j +o(\xi).
\]
In the limit $\xi\rightarrow 0$ the first heavenly system
\begin{align}
  \frac{\partial^2 U}{\partial z^k \partial \tilde{z}^i} \frac{\partial^2 U}{\partial z^l \partial \tilde{z}^j}\eta^{kl} = \frac{\eta_{ij}}{\xi^2}
  \end{align}
reduces to the second Pleba\'nski system (\ref{heavenly2}) to leading order in $\xi$.
In the twistor interpretation $U$ is the generating function for the symplectomorphism between $M$ and $X / \operatorname{ker} \Omega_{+}$, and
$\Theta$ is the infinitesimal generating function.  
\subsection{Twistor fibrations for paraconformal structures}
\label{section21}
The family of compatible metrics associated to a $(2, 2)$-paraconformal structure is a conformal structure in the usual sense.  The condition of anti-self-duality of the Weyl tensor  leads to  Penrose's holomorphic twistor correspondence \cite{P}.  \par 
To summarise the twistor correspondence for  $(2n,2)$-paraconformal structures,  first note that given a compatible metric (\ref{compatible_metric}) and $x \in X$ there is a distinguished cone of null vectors in $T_x\X$ of the form
\be
\label{null}
V=p\otimes q, \quad\mbox{where}\quad p\in \spp_x, q\in \spp'_x.
\ee
Fixing $q$ while varying $p$  the vectors (\ref{null}) span a $2n$--dimensional totally-null subspace of $T_x\X$.  We shall call this an $\alpha$--plane.  
Similarly,  for each fixed $p\in \Gamma(\spp)$ there is a $2$--dimensional totally-null $\beta$--plane obtained by varying $q$
in (\ref{null}).  $\alpha$--planes and $\beta$--planes generalise the classification of null planes into self-dual or anti-self-dual planes for a holomorphic metric in four dimensions; if $n=1$ all null vectors take the form (\ref{null}).  \par
A $2n$--dimensional submanifold is called an $\alpha$--surface if its tangent planes are $\alpha$--planes.  
In the terminology of \cite{BE},  a paraconformal structure is called \textit{right--flat} if there is an $\alpha$--surface tangent to each $\alpha$--plane.  This is equivalent to the anti-self-duality condition if $n=1$.  For general $n$ it is equivalent to integrability of the rank--$2n$ twistor distribution $\mathcal{D}_\alpha$ on $\mathbb{P}(\mathbb{S}')$ that has leaves that then push down to $\alpha$-surfaces.  The $\alpha$--surface tangent to an $\alpha$--plane is unique.  In the right--flat case the we get $\CP^1=\PP(\spp_x')$ worth of $\alpha$--surfaces through each point of $\X$.
This gives rise to a double--fibration picture
\be
\label{fibration}
\X\longleftarrow \PP(\spp')\longrightarrow \Y,
\ee
where $\Y$ is the  twistor space of $\alpha$--surfaces in $\X$. We compute its dimension to be
\begin{eqnarray*}
\dim(\Y)&=&\dim(\X)+\dim(\CP^1)-\mbox{rank}(\D_\alpha) =2n+1.
\end{eqnarray*}
If there is an integrable $\beta$--surface tangent to the $\beta$--plane corresponding to each $p\in \spp_x$, then there exists
another double--fibration picture
\[
\X\longleftarrow \PP(\spp)\longrightarrow {\mathcal W},
\]
where $\mathcal{W}$ is the dual twistor space arising as the quotient of $\PP(\spp)$ by the rank-2 $\beta$--surface distribution $\D_\beta$. If $\mathcal{W}$ exists, then its dimension is
\begin{eqnarray*}
\dim(\mathcal{W})&=&\dim(\X)+\dim(\CP^{2n-1})-\mbox{rank}(\D_\beta)\\
&=& 4n+(2n-1)-2=6n-3.
\end{eqnarray*}
If both $\Y$ and $\mathcal{W}$ exist,  then the paraconformal structure is locally modelled on the Grassmanian of $2$-planes in $\mathbb{C}^{2n+2}$ and there is a compatible metric with vanishing Riemann tensor.  Any compatible metric which is hyper-K\"ahler will be flat.   Note that if $n=1$ and $\X$ is a four--manifold, then both $\Y$ and $\mathcal{W}$ are three--dimensional. If $n>1$ then 
$\dim(\mathcal{W})>\dim(\X)$ which makes $\mathcal{W}$ of limited interest as far as solving equations on $\X$ is 
concerned\footnote{This is the twistor dogma \cite{MW, D2, ADM}: for `good' twistor correspondences,  the dimension of the twistor space equals to the dimension of the Cauchy data on $\X$. The unconstrained \v{C}ech cohomology representatives on the twistor space (the `twistor functions') then correspond,  via the Penrose transform,  to solutions of equations (wave, anti--self--dual Yang--Mills, heavenly, ...) on $\X$.}. 
\subsection{Paraconformal structures from Lax distributions}\label{paraconformal_from_lax}
 Let $\mathbb{S}$ be the vector bundle over $X$ of rank $2n$ trivialised by $\sigma_i$,  $i = 1,...,2n$ and $\mathbb{S}'$ rank-2 and trivialised by $o, \iota$.  
A \textit{Lax distribution} is a
a sub-bundle of $T\mathbb{P}(\mathbb{S}') \cong T(X \times \mathbb{CP}^1)$
which takes the form
\begin{align}\label{Lax_distribution}
L(\lambda) = \operatorname{span}\bigg\{L_i = E_{i0'} + \lambda{E_{i1'}} + f_i \frac{\partial}{\partial \lambda}\bigg\}_{i=1}^{2n}.  
\end{align}
Here $E_{ij'}$ are vector fields on $X$ such that 
$TX = \text{span}\{E_{i0'},  E_{i1'}\}_{i=1}^{2n}$,  the $f_i$ are holomorphic functions on $\mathbb{P}(\mathbb{S}')$ and $\lambda = \frac{u^{1'}}{u^{0'}}$ where  $(u^{0'},u^{1'})$ are coordinates for the fibres of $\mathbb{S'}$ induced by the trivialisation $o, \iota$.  We call the fibre coordinate $\lambda$ the \textit{spectral parameter}.  \par To such a Lax distribution we may associate a $(2n,2)$-paraconformal structure (\ref{paraconeq}) given by 
\[
\sigma_i \otimes o \mapsto E_{i0'}, \quad
\sigma_i \otimes \iota \mapsto E_{i1'}. 
\]
 If $L(\lambda)$ is a Frobenius integrable distribution for each value of $\lambda$ then this paraconformal structure is right-flat.  \par 
Fixing $\epsilon$ by $\epsilon(o, \iota) = 1$ and letting $\{E^{i0'},  E^{i1'}\}$ be the dual sections trivialising $T^*X$ defined by $E^{ik'}(E_{jl'}) = \delta^i_j\delta^{k'}_{l'}$ the compatible metrics (\ref{compatible_metric}) are of the form
\begin{equation}\label{Lax_metric}
g = e_{ij}\epsilon_{i'j'} E^{ii'} \odot E^{jj'} 
\end{equation}
for some non-degenerate skew-matrix of holomorphic functions $e_{ij}$ on $X$.  \par 
In the case $f_i = 0$ for $i = 1,...,2n$ it is natural to define endomorphisms $I,J,K$ of $TX$ by the formulae (\ref{quaternions}) for which the metric $g$ is Hermitian.  Explicitly the endomorphisms are determined by 
\begin{align}\label{frame_quaternions}
J(E_{j0'}) &= -E_{j1'} \hspace{5mm} K(E_{j0'}) = iE_{j1'}  
\end{align}
and the quaternion relations. 
Note that the family of metrics and quaternionic structure does not depend on the choice of basis $L_i$ for the Lax distribution.  
It is also simple to verify that in this setting,  Frobenius integrability of the Lax distribution $L(\lambda)$ for all $\lambda$ implies that the $\pm i$-eigenspaces of the endomorphisms $I,J,K$  will themselves be Frobenius integrable.   
Given a choice of metric,  the associated $2$-forms (\ref{null_2-forms}) are given by 
\begin{align}
\Omega_{+} \,  &= \frac{1}{2}e_{ij}E^{i0'} \wedge E^{j0'} \\
\Omega_{-} \,  &= \frac{1}{2}e_{ij}E^{i1'} \wedge E^{j1'}. 
\end{align}
 \subsection{Hyper-K\"ahler structures from Lax distributions}
We state conditions under which the family of metrics (\ref{Lax_metric}) contains a complex hyper--K\"ahler metric. \par 
In the $n = 1$ case,  if  $f_1 = f_2 = 0$ and the flows of $E_{ij'}$ preserve a holomorphic volume form $\mbox{vol}$ on $\X$ and also 
\[
[L_1, L_2]=0
\]
then the metric
\be
\label{metriclax}
g=\nu^2(E^{10'}\odot E^{21'}-E^{11'}\odot E^{20'}),\quad\mbox{where}\quad \nu^2=\mbox{vol}(E_{10'}, E_{11'}, E_{20'}, E_{21'})
\ee
is complex hyper--K\"ahler \cite{MN}.  \par 
For $n \ge 1$ a condition is given in \cite{B2}.  It can be reformulated as follows: Suppose that $f_i = 0$ for $i = 1,...,2n$ and also
\begin{align*}
[E_{i0'},E_{j0'}]=0
\end{align*} 
for $i,j = 1,...,2n$.  Suppose there exists a symplectic form $\omega$ on $M := X / \operatorname{span}\{E_{i0'}\}$.  Its pull-back to $X$ necessarily takes the form
\begin{align}
\pi^*\omega = \frac{1}{2}e_{ij}E^{i1'} \wedge E^{j1'}
\end{align}
for some non-degenerate skew-matrix of holomorphic functions $e_{ij}$ on $M$.  Then,  the corresponding metric (\ref{Lax_metric})
is complex hyper-K\"{a}hler if
\begin{align}
(\mathcal{L}_{E_{i1'}}\Omega_+)|_{\operatorname{ker} \pi} = 0,  \quad i = 1,...,2n
\end{align}
where $\pi: X \to M$ is the natural projection.  \par
\section{Hyper--Lagrangians}\label{hyper--lagrangian}

\subsection{Hyper--Lagrangian foliations and projectability}
In this section we shall impose an additional condition on complex hyper--K\"ahler structures
which imposes a differential constraint on the associated second Pleba\'nski potential $\Theta$.  If 
$\text{dim}(X)=4$, then this constraint is sufficient to linearise the heavenly equation.
\begin{defn}[Hyper-Lagrangian foliation]
Let $(X, g), I, J, K$ be a complex hyper--K\"ahler structure.  An integrable rank-$2n$ distribution $\mathcal{B} \subseteq TX$ is called a hyper--Lagrangian foliation if $\mathcal{B}$ is Lagrangian with respect to the holomorphic symplectic forms associated to endomorphisms $I, J, K$.  
\end{defn}
A hyper--Lagrangian foliation is equivalent to a sub-bundle $B \subset \spp$ maximally isotropic for $e$ such that the $2n$--dimensional sub-bundle ${\mathcal B}= B\otimes \spp'$ 
of $T\X$ is Frobenius integrable
\[
[\mathcal{B}, \mathcal{B}]\subset {\mathcal{B}}.  
\]
Such a foliation is preserved by the action of $I, J, K$.  
\begin{defn}[Projectable hyper-Lagrangian foliation]
A hyper--Lagrangian foliation is called {\em projectable} if
there exists a rank--$n$ foliation $\mathcal{L}$ of $M = X / \operatorname{ker} N$ such that $d\pi(\mathcal{B}) = \mathcal{L}$.
In this  case $\mathcal{L}$ is Lagrangian in  $M $  with respect to the symplectic form $\omega$. 
\end{defn}  
  The projectability condition imposes additional local restrictions on the complex hyper--K\"ahler manifold: 
That $\mathcal{B}$ pushes down to $M$ is equivalent to 
\be
\label{frobenius}
[\mathcal{B}, \operatorname{ker} N]\subset {\mathcal{B}}+ \operatorname{ker} N.
\ee
As both $\operatorname{ker} N$ and ${\mathcal B}$ are Frobenius integrable, this is equivalent to Frobenius integrability of the distribution ${\mathcal B} + \operatorname{ker} N$. 
The intersection ${\mathcal{B}} \cap \operatorname{ker} N $ is a distribution of rank $n$, spanned by vectors of the form $p\otimes \iota$, where $p\subset{\Gamma{(B)}}$, and $o\subset{\Gamma{(\spp')}}$ defines the 
null--K\"ahler structure.  Therefore the foliation ${\mathcal{B}}+ \operatorname{ker} N$ is of rank $2n+2n-n=3n$.  The leaves are inverse images (with respect to the projection $\pi$) of the leaves of the Lagrangian foliation $\mathcal{L}$ of $M$.  \par

\subsection{Adapted coordinates}
To make the hyper--Lagrangian notion more amenable to calculations take the coordinates $(x^i,y^i)$ so the metric takes the form (\ref{pleb2_form}) and introduce a basis of $4n$ vector fields $E_{ij'}, i=1, \dots, 2n, j'=0,1$
on $\X$ given by 
\begin{align}\label{standard_trivialisation}
E_{i0'}=\frac{\p}{\p y^i},  \quad E_{i1'}=\frac{\p}{\p x^i}+\eta^{jk}\frac{\p^2 \Theta}{\p y^i \p y^j}\frac{\p}{\p y^k}
\end{align}
so that
\begin{align*}
g = \eta_{ij}\epsilon_{i'j'}E^{ii'} \odot E^{jj'}.  
\end{align*}
Given a projectable hyper-Lagrangian foliation we have the coordinate freedom to choose the $x^i$ to be Darboux coordinates adapted to the Lagrangian foliation $d\pi(\mathcal{B}) = \mathcal{L}$ of $M$ in the sense
\begin{align*}
\mathcal{L}  = \operatorname{span}\Big\{\frac{\partial}{\partial x^i}\Big\}_{i=1}^n.  
\end{align*}
The hyper-Lagrangian condition then implies
\begin{align}
\mathcal{B} = \operatorname{span}\big\{E_{i0'}, E_{i1'} \big\}_{i=1}^{n}.  
\end{align}
Next,  we consider the differential consequences of $[\mathcal{B},  \mathcal{B}] \subseteq \mathcal{B}$.  Firstly,  $[E_{i1'},E_{j1'}] = 0$ is already implied by the heavenly equations (\ref{heavenly2}) so the only non-trivial condition is
\begin{align}
[E_{i0'},E_{j1'}] = \eta^{kl}\frac{\p^3 \Theta}{\p y^i \p y^j \p y^k}\frac{\p}{\p y^l} \in \mathcal{B}, \quad  i,j \le n. 
\end{align}
which implies 
\begin{align}\label{pleb2_quadratic}
\frac{\p^3 \Theta}{\p y^i \p y^j \p y^k} = 0, \quad i,j,k \le n. 
\end{align}

If $n=1$, the hyper-Lagrangian foliation is the same as a foliation by $\beta$-planes\footnote{For an anti-self-dual metric,  the leaf space of a $\beta$-plane foliation induced by a null conformal Killing vector was shown to have an induced projective structure (an equivalence class of torsion--free affine connections sharing the same unparametrised geodesics) in \cite{DW}.  This construction was generalised to a larger class of $\beta$-plane foliations in \cite{C}.}.
Altogether this proves 
\begin{thmr}{\ref{theo1_intro}}
The following conditions are equivalent for a complex hyper--K\"ahler manifold $(\X, g)$ 
\begin{enumerate}
\item $\X$ is foliated by projectable hyper--Lagrangian submanifolds.
\item At at each point of $M$ there exists a local coordinate system $(x^i,y^i)$,
such that the second Pleba\'nski potential corresponding to the metric $g$ is at most quadratic in half of the 
fibre coordinates in the fibration $\pi: X\rightarrow M$.
\end{enumerate}
If $n=1$ then the conditions \rm(1)\it and \rm(2)\it \! are equivalent to the existence of a two--parameter family
of anti--self--dual null surfaces in $X$ which push down to curves in $M$.  
\end{thmr} 
\subsection{Discussion}
In \cite{B2v1} Bridgeland defines a \textit{good Lagrangian submanifold} of $M$.  A good Lagrangian $B \subset M$ is a submanifold such that the Lax distribution
$L(\lambda) \subseteq TX$ descends to a distribution on the total space of the normal bundle $NB$ when pushed forward by the map $X_B \to NB$,  where $X_B := \pi^{-1}(B) \subset X$.  An equivalent condition is that in Pleba\'nski coordinates $(x^i,y^i)$ coming from Darboux coordinates adapted to $B$ in the sense $B$ corresponds to holding $x^{n+1},..., x^{2n}$ constant, the equation (\ref{pleb2_quadratic}) holds on $X_B$.  It follows that a projectable hyper-Lagrangian foliation is equivalent to a foliation of $M$ by good Lagrangian submanifolds.  \par 
In \cite{B2v1} it is shown that for the Joyce structures on moduli spaces $M$ of Riemann surfaces of fixed genus equipped with holomorphic quadratic differential,  $M$ admits a foliation by good Lagrangian submanifolds and hence a projectable hyper-Lagrangian foliation of $X = TM$.  The leaves in $M$ are spaces of quadratic differentials holding the Riemann surface fixed.  The hyper--K\"ahler metrics we construct in \S \ref{A_N metric} (which includes as a special case the metric of \cite{BM}) will also be shown to admit projectable hyper-Lagrangian foliations.  \par

\section{Joyce structures}\label{joyce}
Joyce structures were introduced in \cite{B3}, with a refined definition (which we shall follow below) given in \cite{BS}. Let $(M,\omega)$ be a complex manifold of complex dimension $2n$ with holomorphic symplectic form $\omega$.   A lattice $\mathcal{G}_p$ at a point is a discrete subgroup of $T_pM$ such that complex multiplication $\mathcal{G}_p \otimes_{\mathbb{Z}} \mathbb{C} \to TM$ is an isomorphism.  A holomorphically varying bundle of lattices $\mathcal{G} \hookrightarrow TM$ defines a linear connection $\nabla$ on the bundle $TM$ by picking out a $2n$-dimensional vector space of parallel sections.  Specifically,  these are taken to be the sections which take values in the lattice.  \par The first ingredient of a Joyce structure on $M$ is such a lattice on which the symplectic form takes rational values,  and a compatible $\mathbb{C}^*$-action,  or equivalently a holomorphic vector field $W_0$ such that $\nabla W_0$ is the identity endomorphism.  It follows that there are canonical complex coordinates $\zeta^i$ on any suitably small neighbourhood of $M$ in which 
\begin{equation}\label{joyce_homothety}
W_0 = \zeta^i\frac{\partial}{\partial \zeta^i}
\end{equation}
and 
\begin{equation}
\omega = \frac{1}{2}\eta_{ij}d\zeta^i \wedge d\zeta^j. 
\end{equation}
We will call these coordinates \textit{period coordinates} for $M$.  \par 
 Next,  the \textit{non-linear data} of a Joyce structure is a compatible complex hyper-K\"ahler structure on the complex manifold $X := TM$,  along with some symmetry conditions.  \par 
In the notation introduced in the previous section the compatibility just means we can take the  metric $g$ to take the second heavenly form (\ref{pleb2_form}) in the period coordinates $\zeta^i$ and the associated fibre coordinates $\theta^i$.  Then $\Omega_- = \pi^*\omega$.  We will also call $(\zeta^i,\theta^i)$ period coordinates (for $X$).  \par  
The symmetry conditions are that the second Pleba\'nski potential in the period coordinates (which we will call $\Phi$) can be taken to satisfy: 
\begin{enumerate} 
\item  $\mathcal{L}_Wg = g$ where $W$ is the $\nabla$-horizontal lift of $W_0$.  \label{homogeneity}  
\item $\Phi$ is odd under simultaneously multiplying all the fibre coordinates by $-1$.   \label{oddity}
\item $\Phi$ is invariant under translations of the fibres by the lattice.  \label{periodicity} 
\end{enumerate}
 Condition (\ref{homogeneity}) is equivalent to the function $\Phi$ being homogeneity $-1$ under simultaneous rescalings of the $\zeta^i$ since $W = \zeta^i$\large$\frac{\partial}{\partial \zeta^i}$\normalsize , is the coordinate expression for $W$ in the period coordinates $(\zeta^i,\theta^i)$.  
%
\section{Reductions of the heavenly equations}\label{reductions}
\subsection{Projectability condition}
\begin{thmr}{\ref{theo2_intro}}
 Suppose a complex hyper--K\"ahler manifold $(\X, g)$ of complex dimension four is foliated  by projectable hyper--Lagrangian submanifolds. Then in the adapted coordinates of Theorem \rm{\ref{theo1_intro}} \it the heavenly equation for the Pleba\'nski potential reduces to a system which is linearisable by a contact transformation.  
 \end{thmr}
\begin{proof}
Given a hyper-K\"ahler metric on a complex four-manifold $X$ take coordinates 
$(z,  w,  x,  y)$ adapted to 
Pleba\'nski's second heavenly equation so that
\begin{eqnarray}
\label{heavy}
&&\Theta_{wx}-\Theta_{zy}-\Theta_{xx}\Theta_{yy}+\Theta_{xy}^2=0, \quad\mbox{and}\\ 
&&g=dwdx-dzdy+\Theta_{xx}dz^2+2\Theta_{xy}dzdw+\Theta_{yy}dw^2.\nonumber
\end{eqnarray}
From Theorem \ref{theo1_intro} given $X$ admits a projectable hyper-Lagrangian foliation we may take $\Theta_{xxx}=0$, or
\begin{equation}\label{quad_decomp}
\Theta=Ax^2+2Bx+C,
\end{equation}
where $(A, B, C)$ are functions of $(z,w,y)$ which satisfy
\begin{eqnarray}
&&2AA_{yy}-4{A_y}^2+A_{yz}=0, \label{Aquad} \\
&&2AB_{yy}-4B_yA_y+B_{yz}-A_w=0, \label{Bquad} \\
&&2AC_{yy}-4{B_y}^2+C_{yz}-2B_w=0 \label{Cquad}.
\end{eqnarray}
Given a general solution $A(z, w, y)$ the remaining two equations are linear.  We shall show how to linearise the first equation by a Legendre transformation,  assuming $A_{yy}$ is non-vanishing on an open region.  Rewrite this
equation as a differential ideal
\begin{eqnarray*}
T_1&\equiv& dA-A_ydy-A_zdz-A_wdw=0,\\
T_2&\equiv& (2 A dz\wedge d A_y-4 {A_y}^2dz\wedge dy+dA_y\wedge dy)\wedge dw=0.
\end{eqnarray*}
We use $p=A_y$ as a coordinate and set $F(z,w,p)=A-py$ so that $T_1=0$ becomes
\begin{eqnarray*}
dF&=&A_zdz+A_wdw-ydp\\
&=&F_zdz+F_wdw+F_pdp
\end{eqnarray*}
which yields $A_z=F_z,  A_w=F_w, y=-F_p$.  Substituting this into $T_2=0$ gives a linear PDE
\begin{equation}
2(F-pF_p)+4p^2F_{pp}+F_{pz}=0.
\end{equation}
\end{proof}
\label{ellipticity}
\rm

\begin{rem}
\rm
In the case $(X,g)$ is the hyper-K\"ahler structure for a Joyce structure on $M$ then $A$ and $C$ must be odd functions of $y$,  while $B$ even.  \par 
\end{rem}
For a Joyce structure with projectable hyper-Lagrangian foliation,  there are two different types of useful Pleba\'nski coordinates: those adapted to the foliation,  and the period coordinates.  Seeing the hyper-Lagrangian foliation in the period coordinates and vice versa is not so easy.  This motivates the following:
\begin{prop}\label{elliptic_prop}
Suppose a complex hyper--K\"ahler manifold $(\X, g)$ of complex dimension four corresponds to a Joyce structure on $M$.  Given Pleba\'nski coordinates $(z,w,x,y)$ such that the decomposition \rm{(\ref{quad_decomp})} \it holds,  $A_y$ is an even elliptic function of $y$.  
\end{prop}
\proof
The coordinates $(x^1,x^2,y^1,y^2) := (z,w,x,y)$ and the period coordinates $(\zeta^1, \zeta^2,  \theta^1,  \theta^2)$ are related by the lift of the canonical transformation between $(x^1,x^2)$ and $(\zeta^1,\zeta^2)$ to $X$.  
The potential will be modified by some function cubic in the fibre coordinates.  Write
\begin{align}\label{pleb_change}
\Phi(\zeta^1, \zeta^2, \theta^1,  \theta^2) &=  \Theta(x^1,x^2,y^1,y^2) + D_{ijk}y^iy^jy^k \\ 
&= 
Ay^1y^1+2By^1+C + D_{ijk}y^iy^jy^k \nonumber
\end{align}
where $\Phi$ is the Pleba\'nski potential corresponding to the period coordinates and $D_{ijk}$ are functions on $M$.  Differentiate (\ref{pleb_change}) by $y^1$ twice and $y^2$ once.  Since the two sets of fibre coordinates are related by linear transformations depending on the base the result is
\begin{align*}
\frac{\partial \zeta^i}{\partial x^1}\frac{\partial \zeta^j}{\partial x^1}\frac{\partial \zeta^k}{\partial x^2}\frac{\partial^3 \Phi}{\partial \theta^i \partial \theta^j  \partial \theta^k} = 2A_{y^2} +  2D_{112}.  
\end{align*}
The partial derivatives of $\Phi$ taken with respect to the period coordinates will be invariant under lattice transformations,  while the Jacobian matrices and $D_{112}$ are functions on $M$.  Accordingly,  $A_{y^2}$ will be invariant under lattice transformations 
\begin{align*}
\zeta^i \mapsto \zeta^i + L_{j}{}^{i}v^j
\end{align*}
where $v^i \in \mathbb{Z}$ and $L_{j}{}^{i}$ defines a non-degenerate lattice.  This corresponds to
\begin{align}\label{lattice_transform}
y^i \mapsto y^i + \frac{\partial{x^i}}{\partial \zeta^j}L_{k}{}^{j}v^k.  
\end{align}
So $A_{y^2}$ is doubly periodic in $y^2$ with periods
\begin{align*}
\omega_1 := \frac{\partial{x^2}}{\partial \zeta^i}L_{1}{}^{i},  \quad  \omega_2 := \frac{\partial{x^2}}{\partial \zeta^i}L_{2}{}^{i}
\end{align*}
which are functions on $M$.  \par Fix the point on $M$.  
If the periods are linearly independent over $\mathbb{R}$ then $A_{y^2}$ is an elliptic function of $y^2$.  
Otherwise,  in the degenerate case where these are linearly dependent over $\mathbb{R}$ then,  we can find a sequence of transformations (\ref{lattice_transform}) defined by $k_{(m)}^i \in \mathbb{Z}$ such that $\omega_jk^j_{(m)} \to 0$ and $\big|\frac{\partial{x^1}}{\partial \zeta^i}L_{j}{}^{i}k^j_{(m)}\big| \to \infty$.  That (\ref{pleb_change}) is invariant under this sequence of transformations implies the right hand side has no polynomial dependence on $y^1$.  This implies $A = -D_{112}y^2$ so $A_{y^2} = -D_{112}$ is constant (so elliptic) with respect to $y^2$.  \par
Since $A$ is odd,  $A_{y^2}$ will be an even function of $y^2$.  
\qed 

\subsection{Strengthened projectability condition} The projectability condition on the hyper--Lagrangian foliation $\mathcal{B}$ is that $\mathcal{B}$ pushes down to the quotient by the $\alpha$-plane foliation given by the null structure (\ref{onull}).  The hyper--Lagrangian foliation is preserved by the action of $I, J, K$.  In the $n=1$ case,  if we insist that $\mathcal{B}$ pushes down to the quotients corresponding to the whole family of $\alpha$-plane foliations then we can solve the (first) heavenly equation completely.  
This condition in fact makes sense for any $n$ and is the condition for $\mathcal{B}$ to define a rank-$2n$ foliation of the twistor space $\Y$.  
\begin{prop}
Given a complex hyper-K\"ahler metric on a complex four-manifold $X$,  if $X$ admits a hyper-Lagrangian foliation such that for any $\alpha$-plane foliation $\mathcal{D}_\alpha$ we have
\begin{align*}
[\mathcal{B},  \mathcal{D}_\alpha] \subseteq \mathcal{B} + \mathcal{D}_\alpha
\end{align*}
then there exist coordinates $(u^1,u^2,\tilde{u}^1,\tilde{u}^2)$ such that
\begin{equation}
\label{Timmetric}
g = A \, du^1 d\tilde{u}^2 + \frac{1}{A} \,  du^2 d\tilde{u}^1 + \bigg(u^1\frac{\partial A}{\partial u^2} + \tilde{u}^1A^{-2}\frac{\partial A}{\partial \tilde{u}^2} + D  \bigg)du^2 d\tilde{u}^2
\end{equation}
for arbitrary holomorphic functions $A(u^2,\tilde{u}^2),D(u^2,\tilde{u}^2)$.   
\end{prop}
\begin{proof}
As before take Pleba\'nski coordinates $(x^1,x^2,y^1,y^2)$ adapted to the Lagrangian foliation of $M$ so that in the trivialisation (\ref{standard_trivialisation}) 
$$
\mathcal{B} = \operatorname{span}\big\{E_{10'}, E_{11'} \big\}. 
$$
Recalling the relationship between the coordinates $(x^1,x^2,y^1,y^2)$ and $(z^1,z^2,\tilde{z}^1,\tilde{z}^2)$ adapted to the first and second Pleba\'nski equations respectively,  we have
\begin{align*}
\frac{\p}{\p y^i} = \eta^{jk}\frac{\partial^2 U}{\p z^i \p \tilde{z}^j}\frac{\p}{\p \tilde{z}^k}
\end{align*}
and so
\begin{equation}\label{theta_ABC}
\frac{\partial^3 \Theta }{\partial y^i \partial y^j \partial y^k} = \frac{\partial^2 U}{\partial z^i \partial \tilde{z}^l} \eta^{lp} \frac{\partial ^3 U}{\partial \tilde{z}^p \partial z^j \partial z^k}.  
\end{equation}
The strengthened projectability condition implies $[\mathcal{B},  \mathcal{D}] \subseteq \mathcal{B} + \mathcal{D}$ where $\mathcal{D}$ is the $\alpha$-foliation 
$\mathcal{D} = \text{span}\{E_{11'},E_{21'}\}$.  Then
\begin{align*}
[E_{10'},E_{i1'}] =\eta^{kl}\frac{\p^3 \Theta}{\p y^i \p y^1 \p y^k}\frac{\p}{\p y^l} \in \mathcal{B} + \mathcal{D}, \quad i = 1,2
\end{align*}
which implies 
\begin{align*}
\frac{\p^3 \Theta}{\p y^i \p y^1 \p y^1}=0, \quad i = 1,2.  
\end{align*}
So (\ref{theta_ABC}) then implies 
\[
 \frac{\partial ^3 U}{\partial \tilde{z}^i \partial z^1 \partial z^1} = 0,  \quad i = 1,2  
\]
so that $U = A(z^2,\tilde{z}^1,\tilde{z}^2)z^1 + B(z^2,\tilde{z}^1,\tilde{z}^2) + C(z^1,z^2)$.  Substituting this ansatz into the first heavenly equation yields two equations
\begin{align}
\frac{\partial A}{\partial \tilde{z}^1}\frac{\partial^2 A}{\partial z^2 \partial \tilde{z}^2} - \frac{\partial A}{\partial \tilde{z}^2} \frac{\partial^2 A}{\partial z^2 \partial \tilde{z}^1} &= 0 \\
\label{Bsys} \frac{\partial A}{\partial \tilde{z}^1}\frac{\partial^2 B}{\partial z^2 \partial \tilde{z}^2} - \frac{\partial A}{\partial \tilde{z}^2} \frac{\partial^2 B}{\partial z^2 \partial \tilde{z}^1} &= 1. 
\end{align}
Interchanging $\tilde{z}^1$ and $\tilde{z}^2$  if necessary,  the first equation is equivalent to the existence of function $F(\tilde{z}^1,\tilde{z}^2)$ such that 
\begin{equation} \label{Asys}
\frac{\partial A}{\partial z^1} = F(\tilde{z}^1,\tilde{z}^2)\frac{\partial A}{\partial \tilde{z}^2}.
\end{equation}
Now introduce coordinates $\tilde{u}^1(\tilde{z}^1,\tilde{z}^2),   \tilde{u}^2(\tilde{z}^1,\tilde{z}^2), u^1 = z^1,  u^2 = z^2$ arranged so that 
\[
\frac{\partial}{\partial \tilde{u}^1} = \frac{\partial}{\partial \tilde{z}^1}  - F(\tilde{z}^1,\tilde{z}^2)\frac{\partial}{\partial \tilde{z}^2}.  
\]
The form of the metric 
(\ref{pleb1_form})  is preserved.  The system of PDEs (\ref{Bsys},  \ref{Asys}) becomes 
\begin{align*}
\frac{\partial A}{\partial \tilde{u}^1} &= 0 \\ 
\frac{\partial B}{\partial \tilde{u}^1 \partial u^2} &= \Bigg(\frac{\partial A}{\partial \tilde{u}^2}\Bigg)^{-1}.  
\end{align*}
The general solution for $U$ is then
\be
U = A(\tilde{u}^2,u^2)u^1 + \int \frac{\tilde{u}^1}{\frac{\partial A}{\partial \tilde{u}^2}} \, du^2 + C(u^1,u^2) + D(\tilde{u}^2,u^2) + E(\tilde{u}^1,\tilde{u}^2).  
\ee
Differentiating gives the metric  (\ref{Timmetric}).  \end{proof} 
\subsection{The null-case} 
For a Joyce structures with $n = 1$,  the condition that the homothety is a null-vector implies the metric takes an explicit form generalising that of Sparling and Tod \cite{ST}.  
\begin{prop}
Given a hyper-K\"ahler metric on a complex four-manifold $X$ in the form \rm{(\ref{pleb2_form})} \it with homothety $W = x^i$\Large $\frac{\partial}{\partial x^i}$ \normalsize satisfying $\mathcal{L}_Wg = g$ and $g(W,W) = 0$,  the metric is given by 
\begin{equation}\label{sparling-tod}
g = \eta_{ij} dy^{i} \odot dx^j +  \eta_{kp}\eta_{lq} \frac{F\big(\frac{x^1}{x \cdot y},\frac{x^2}{x \cdot y}\big)}{(x \cdot y)^3} x^kx^l \, dx^p \odot dx^q
\end{equation}
where $F$ is an arbitrary holomorphic function of two variables and $x \cdot y = \eta_{ij}x^iy^j$.  
\end{prop}
\begin{proof}
We work locally away from the hypersurface where $x \cdot y = 0$.  Then,  we may use the fact that $x^ix^j$,  $x^{(i}y^{j)}$ and $y^iy^j$ span the space of symmetric rank--two spinors.  $g(W, W)$ does not vanish unless we have the following decomposition with respect to this basis:
\begin{equation}\label{hessian}
\frac{\partial^2 \Theta}{\partial y^i \partial y^j}  = \eta_{ki}\eta_{lj}(Gx^{(k}y^{l)} + Hx^kx^l)
\end{equation}
for some holomorphic functions $G$ and $H$ on $X$.  
The homothety condition is that the above expression has homogeneity $-1$ with respect to rescalings by the $x^i$ coordinates.  Accordingly 
\be
\label{G_homo}
W(G) = -2G, \quad
W(H) = -3H.
\ee
Differentiating (\ref{hessian}) by $y^k$ and enforcing that the result is symmetric on all indices implies
\be
y^i\frac{\partial G}{\partial y^i} + 2x^i \frac{\partial H}{\partial y^i}  = -3G, \quad
x^i \frac{\partial G}{\partial y^i} = 0.   
\label{G_is_killed}
\ee
The second heavenly equation yields 
\begin{equation*}
\eta^{ij}\frac{\partial^2 \Theta}{\partial y^i \partial x^j} = \frac{1}{4}G^2 (x \cdot y)^2
\end{equation*}
which we differentiate by the fibre coordinates to get 
$$
\frac{1}{2}\frac{\partial G}{\partial x^j}x^jy^i + \frac{1}{2}\frac{\partial G}{\partial x^j}y^jx^i + \frac{3}{2}Gy^i + \frac{\partial H}{\partial x^j}x^jx^i + 3Hx^i - \frac{1}{2}G\eta^{ij}\frac{\partial G}{\partial y^j}(x \cdot y)^2 - \frac{1}{2}G^2x^i(x \cdot y) = 0
$$
Contracting the above with $\eta_{ik}x^k$ and using (\ref{G_homo}) and (\ref{G_is_killed}) implies $G = 0$.  We thus have
\begin{equation}\label{TisKilling}
x^i \frac{\partial H}{\partial y^i} = 0.
\end{equation}
 Since $H$ is homogeneity $-3$ in the $x^i$ coordinates we can write:
\begin{equation}\label{H_form}
H = \frac{F\Big(\frac{x^1}{x \cdot y}, \frac{ x^2}{x \cdot y}, y^1\Big)}{(x \cdot y)^3}.
\end{equation}
for some function $F$.  Write $F_3$ to denote the derivative of the function $F$ by its third argument.  Differentiating (\ref{H_form}) using the chain rule yields,  after a short calculation
$$
x^i \frac{\partial}{\partial y^i} \  F\bigg(\frac{x^1}{x \cdot y}, \frac{ x^2}{x \cdot y}, y^1\bigg) = x^1F_3\bigg(\frac{x^1}{x \cdot y}, \frac{ x^2}{x \cdot y}, y^1\bigg),  
$$
but the left hand side vanishes due to (\ref{TisKilling}) so $F$ does not depend on its third argument.  Substituting $H$ back into (\ref{hessian}) gives the metric components.  
\end{proof}
Any metric of the form (\ref{sparling-tod}) is complex hyper-K\"ahler and it admits a Killing vector
\begin{equation}
T := x^i\frac{\partial}{\partial y^i}.
\end{equation}
The metric of \cite{ST} takes $F$ to be constant.  In this case it admits a projectable hyper-Lagrangian foliation $\operatorname{span}\{x^iE_{i0'},  x^iE_{i1'}\}$.  \par 
\begin{rem}\rm
It is possible via a similar calculation to show if $n=1$ that the metric (\ref{sparling-tod}) is the general complex hyper-K\"ahler metric in the form (\ref{pleb2_form}) with both sides of the second heavenly equation (\ref{heavenly2}) vanishing separately and with homothety $W= x^i$\Large $\frac{\partial}{\partial x^i}$\normalsize.  
\end{rem}
\section{The deformed cubic oscillator, and  Painlev\'e I}
\label{A2andD21}
In \cite{BM} a complex hyper-K\"ahler metric was constructed via a Lax pair of isomonodromic flows for an isomonodromy problem for a Schr\"odinger equation with deformed cubic oscillator potential.  This calculation leads to the $A_2$ Joyce structure on the space $M = \text{Quad}(0,\{7\})$.  The reason for this name is that the moduli space realises the space of stability conditions of a $CY_3$ triangulated category associated to the $A_2$ quiver.  Furthermore the base $M$ can be identified with unfolding of the $A_2$ singularity $y^2 = x^3$ in the ADE classification.  \par 
In this section we shall explicitly construct the complex hyper-K\"ahler metric,  and analyse its singularities.
Our tool will be the Lax pair characterisation of the hyper--K\"ahler metric (\ref{metriclax}).
\subsection{Deformed cubic oscillator}\label{pflow1}
Consider the differential equation 
\begin{equation}\label{schro}
y'' = Q(x)y
\end{equation}
where
\begin{equation}\label{potential}
Q(x) := \frac{Q_0(x)}{\lambda^2} + \frac{Q_1(x)}{\lambda} + Q_2(x)
\end{equation}
is a deformation of the polynomial cubic oscillator potential given by 
\begin{align*}
Q_0(x) &= x^3 + ax + b \phantom{\frac{r^2}{p^2}}\\
Q_1(x) &= \frac{p}{x-q} + r \phantom{\frac{b}{p^2}} \\
Q_2(x) &= \frac{3}{4(x-q)^2} + \frac{r}{2p(x-q)} + \frac{r^2}{4p^2}.
\end{align*}
 The corresponding isomonodromy problem was the one studied by Bridgeland and Masoero in \cite{BM} and is a modification to the isomonodromy problem corresponding to the $PI$ Hamiltonian studied by Okamoto \cite{O}.  \par
Let the four--manifold $\X$ be defined as a submanifold
of $\C^5$ with coordinates $(a, b, q, p, r)$
given by
\be
\label{cubic}
 p^2 = q^3+aq+b.
\ee
Points in the space $X$ specify potentials of the form (\ref{potential}) and there is a fibration $\pi: X \to M$ over the complex manifold $M = \text{Quad}(0,\{7\})$ with coordinates $(a,b)$ corresponding to quadratic differentials $Q_0(x) dx^2$.  The period data of the Joyce structure is constructed via an identification (the \textit{Abelian holonomy map}) of the space $X$ with a discrete quotient of the holomorphic tangent bundle $TM$.  We focus on the metric and homothety here.  \par
The generators of the isomonodromic deformations of (\ref{schro}) were calculated in \cite{BM}.  We will revisit this calculation later in \S \ref{isomonodromy},  but for now,  using $(a,b,q,r)$ as local coordinates,  all we need to know is that they are given by vector fields
\begin{eqnarray}
\label{distribution1}
l_1&=& - \frac{\partial}{\partial r} + \lambda\bigg(\frac{\partial}{\partial b}+\frac{r}{2p^2}\frac{\partial}{\partial r}\bigg) \\
l_2&=& -{2p}\frac{\partial}{\partial q} + \lambda \bigg(\frac{\partial}{\partial a}-q\frac{\partial}{\partial b}-\frac{r}{p}\frac{\partial}{\partial q}-
\frac{r^2(3q^2+a)}{2p^3}\frac{\partial}{\partial r}\bigg). \nonumber
\end{eqnarray}
In particular,  the distribution takes the form (\ref{Lax_distribution}) where we recognise $\lambda$ as the spectral parameter and $f_1=f_2=0$.  
These flows
satisfy $[l_1, l_2]=0$, and moreover finding the integral curves of $l_2$
reduces  to solving the PI.  Parametrise the flow so $\dot{a} = -\lambda$ then
\begin{eqnarray*}
\dot{q}&=&\frac{2p^2+\lambda r}{p}\\
\dot{r}&=&\frac{\lambda r^2(3q^2+a)}{2 p^3}
\end{eqnarray*}
while differentiating the constraint and using the above yields
\begin{equation}
\dot{p} =\frac{6p^2q^2+3\lambda q^2r +2ap^2+\lambda ar }{2p^2}.
\end{equation}
Differentiating the first of these ODEs, and eliminating 
$(\dot{p}, \dot{r})$ using the remaining two ODEs yields a rescaled form of Painlev\'e I.
\be
\ddot{q}=6q^2+2a.
\ee
This calculation generalises that of \cite{BM}, where it was shown how Painlev\'e I arises from (\ref{distribution1}) in the
case where $r=0$.
\subsection{The metric}\label{a_2_joyce_subsection}
For the purposes of writing down the hyper-K\"ahler metric it is convenient to work with the basis $L_1=l_1, L_2=l_2+ql_1$ so
\[
L_1=-\frac{\partial}{\partial r} + \lambda\bigg(\frac{\partial}{\partial b}+\frac{r}{2p^2}\frac{\partial}{\partial r}\bigg)
\quad
L_2= -q\frac{\partial}{\partial r} -2p\frac{\partial}{\partial q} + {\lambda}\bigg(\frac{\partial}{\partial a}-\frac{r}{p}\frac{\partial}{\partial q}-
\frac{r(3q^2r+ar-qp)}{2p^3}\frac{\partial}{\partial r}\bigg).  
\]
These flows commute and preserve the volume form 
\begin{align}
\text{vol} := \frac{1}{2p}db \wedge da \wedge dq \wedge dr.  
\end{align}
and so the conformal class of metrics distinguished by the Lax distribution contains a complex hyper-K\"{a}hler metric.  
Reading off from $L_1,  L_2$ above the tetrad in (\ref{Lax_distribution}) we can use the formula (\ref{metriclax}) for the hyper-K\"ahler metric.  We have $\nu^2 = 1$,  so calculating the dual frame of one--forms $E^{ij'}$,  the complex hyper-K\"ahler metric is
\be
\label{PI1}
g=\Big(\frac{r(3q^2r+ar-2qp)}{2p^3}da-\frac{r}{2p^2}db-\frac{q}{2p}dq+dr\Big)\odot da
-\Big(\frac{r}{2p^2}da+\frac{1}{2p}dq\Big) \odot db.
\ee
The vector field
\begin{equation}
W=\frac{4a}{5}\frac{\partial}{\partial a}+\frac{6b}{5}\frac{\partial}{\partial b} + \frac{2q}{5}\frac{\partial}{\partial q}+\frac{r}{5}\frac{\partial}{\partial r}
\end{equation}
satisfies $\mathcal{L}_Wg = g$ and $g$ does not admit any other Killing vectors, or conformal Killing vectors.\par 
To investigate the singular nature of the hyper-surface $p=0$ in $\X$ consider a local diffeomorphism of $\X$, and use $(a, p, q, r)$ as coordinates (so that $b$ is eliminated using (\ref{cubic})). This puts the metric in the form 
\begin{eqnarray}
  \label{58met}
g&=&\Big(\frac{r(3q^2r+ar-qp)}{2p^3}da-\frac{r}{p}dp+
\frac{3q^2r+ar-pq}{2p^2}dq+dr\Big)\odot da \\
&&-
\Big(\frac{r}{2p^2}da+\frac{1}{2p}dq\Big)\odot(2pdp-qda-(a+3q^2)dq).\nonumber
\end{eqnarray}

Computing the determinant of the quadratic form given by $g$ gives an expression which is regular at $p=0$. We therefore turn to the curvature
invariants. The simplest of these is the squared norm of the Weyl
tensor of $g$.  A cumbersome computation gives
\[
|\mbox{Weyl}|^2=96\frac{(3q^2+a)^2}{p^6}.
\]
This implies the singularity at $p=0$ is not removable.
\vskip 5pt
Letting $\omega = db \wedge da$ define the symplectic structure on $M$,  the metric has the property that $\Omega_{-} = \pi^*\omega$ where $\text{span}\{\partial_r,  \partial_b\}$ is a projectable hyper-Lagrangian foliation\footnote{A calculation verifies that this $\beta$-foliation does not meet the differential condition of \cite{C} for the correspondence of $\beta$-foliations with projective surfaces to apply.}.   Quadratic coordinates of Theorem \ref{theo1_intro} will be obtained as Pleba\'nski coordinates for $X$ corresponding to the Darboux coordinates $(a, b)$ for $M$.  They are obtained in \cite{BM} as follows:
With each point of $M$ associate an elliptic curve 
\[
\Sigma_{(a,b)}=\{(\mu, \zeta)\in \C^2, \mu^2=\zeta^3+a\zeta+b\}.
\]
Then define $(z, w, x, y)$ on $X$ by
\begin{equation}\label{a_2_pleb}
w=a, \quad z=b,  \quad x=-r-\int_{(q, -p)}^{(q, p)} \frac{\zeta d\zeta}{\mu}, \quad y=-\frac{1}{4}\int_{(q, -p)}^{(q, p)} 
\frac{d\zeta}{\mu}.
\end{equation}
where $\mu=\sqrt{\zeta^3+a \zeta+b}$,  and 
the path of integration between the two points $(q, p)$ and $(q, -p)$ in $\Sigma_{(a,b)}$  is invariant under the involution
$(\zeta, \mu)\rightarrow (\zeta, -\mu)$.  \par
From a calculation with the MAPLE \texttt{DifferentialGeometry} package we obtain the metric components in this frame and hence the second derivatives of $\Theta$.  We omit the full expressions which are long and in terms of elliptic and related functions but by integrating we can indeed take $\Theta$ such that  
\be
\label{ThetaABC}
\Theta =Ax^2+2Bx+C.  
\ee
In light of Proposition \ref{elliptic_prop},  we note $A_y$ has a simple expression in terms of Weierstrass elliptic functions
\begin{align}
A_y(z,w,y) = -\frac{1}{\wp'(y;-4z,-4w)^2}.  
\end{align}
\vskip 5pt
The linear systems governing the remaining five
Painlev\'e equations \cite{O} can be similarly modified so that (conformal classes of) metrics can be read off from the isomonodromic flows.   In these cases the
constraint (\ref{cubic}) needs to be replaced by a constraint coming from the 
relevant Painlev\'e Hamiltonian from the Table (H) in \cite{O} (Note that in our case the variable $b$ is just the
Hamiltonian $H_I$).  This work has been carried out in the case of PII and PIII \cite{BDM}.  
\subsection{Series expansions}
In this subsection we shall independently show how to construct the metric (\ref{PI1}) starting directly from the fact the heavenly potential must take the form (\ref{ThetaABC}) in the Pleba\'nski coordinates (\ref{a_2_pleb}),  along with the additional homothety condition.  In this coordinate system the homothety is given by\footnote{The general Euler vector field $W=c_1 w\p_w+c_2 z \p_z$ on
$M$, where $(c_1, c_2)$ are constants, lifts to a homothety of the hyper--K\"ahler metric (\ref{heavy}) of the form
\[
W=c_1 \Big(w\p_w-\frac{1}{3} x\p_x+\frac{2}{3} y\p_y\Big)+c_2
 \Big(z\p_z+\frac{2}{3} x\p_x-\frac{1}{3} y\p_y\Big)+\frac{1}{3}c_3\Big(x\p_x+y\p_y\Big).
\]
Here $c_3$ is another constant, and $\Theta$ is constrained by the condition $W(\Theta)=c_3 \Theta$. The three constants $(c_1, c_2, c_3)$ are
defined up to an overall constant scaling.}  
\[
W=\frac{4w}{5}\frac{\p}{\p w}+\frac{6z}{5} \frac{\p}{\p z}+\frac{x}{5}\frac{\p}{\p x}-\frac{y}{5}\frac{\p}{\p y}, \quad\mbox{and}\quad W(\Theta)=-\Theta.
\]
The quadraticity, homogeneity, and the parity imply
\[
W(A) = -\frac{7}{5}A, \quad W(B)=-\frac{6}{5}B, \quad W(C)=-C, 
\]
where $A, B, C$ are functions of $(w, z, y)$. Therefore
\[
A=\sum_{k=1}^{\infty} a_{2k+1}y^{2k+1}, \quad B=\sum_{k=2}^{\infty} b_{2k} y^{2k}, \quad C=\sum_{k=2}^{\infty}c_{2k+1}y^{2k+1}
\]
where
\begin{eqnarray*}
  a_3&=&w^{-1}\alpha_3(t), \quad a_5=wz^{-1} \alpha_5(t), \quad a_7=\alpha_7(t), \quad a_9=w^{-1}z\alpha_9(t), \quad
         a_{11}=w\alpha_{11}(t), \quad\dots\\
  b_4&=&wz^{-1} \beta_4(t), \quad b_6=\beta_6(t), \quad
         b_8=w^{-1}z\beta_8(t),\quad b_{10} =w\beta_{10}(t), \quad\dots\\ 
c_5&=&\gamma_5(t), \quad c_7=w^{-1}z\gamma_7(t), \quad c_9= w \gamma_9(t), \quad c_{11}= z\gamma_{11}(t), \quad\dots\;\;.\\   
\end{eqnarray*}
The invariant coordinate $t=w^{3}z^{-2}$ satisfies $W(t)=0$, and the coefficients $(\alpha, \beta, \gamma)$ can be found by solving a recursive set of ODEs resulting from (\ref{Aquad}), (\ref{Bquad}), and  (\ref{Cquad}). The particular solution corresponding to the metric (\ref{PI1}) is
\begin{eqnarray}
\label{Toml}
A&=&-\frac{1}{28}y^7+\frac{1}{110}wy^{11}+\frac{1}{91}zy^{13}-\frac{1}{300}w^2y^{15}+\dots,\\
B&=& \frac{1}{30}y^6-\frac{1}{150}wy^{10}-\frac{23}{2310} zy^{12}+\frac{127}{47775}w^2 y^{14}+\dots \nonumber\\
C&=&  -\frac{1}{30}y^5+\frac{17}{3780}wy^9+\frac{52}{5775} zy^{11}-\frac{323}{150150}w^2 y^{13}+\dots\;\;.   \nonumber
\end{eqnarray}
\section{Isomonodromy of the deformed polynomial oscillator}\label{isomonodromy}
Motivated by the $A_2$ case,  in this section we construct the isomonodromic flows for the equation (\ref{schro}) taking the potential (\ref{potential}) to be a deformed polynomial oscillator potential of odd-degree.  \par 
We will  suppress the summation convention in this section.  
\subsection{Isomonodromic flows}
The solutions of linear ODE (\ref{schro}) with meromorphic potential may have branching behaviour near a set of poles $x_0,...,x_M$.  The fundamental group of the punctured space $\mathbb{CP}^1 \setminus \{x_0,...,x_M\}$ then has a linear representation on the space of solutions called the \textit{monodromy}.   \par 
As in \S \ref{A2andD21} we will consider a complex manifold $X$ parametrising a family of ODEs with deformed polynomial oscillator potential.  
By \textit{isomonodromic flow} we will mean a vector field on $X$ generating a one-parameter group of deformations which preserves the monodromy.  In our case,  the ODE (\ref{schro}) will have an irregular (non-Fuchsian) singular point and so strictly speaking,  by monodromy we mean the \textit{partial monodromy data} consisting of the connection matrices and Stokes' multipliers as defined in \cite{M}.  The only part of the theory we will really utilise is the sufficient and necessary condition for a flow to be isomonodromic.  This  is a generalisation \cite{U,  JMU} of Schlesinger's result for ODE with regular singular points. \par 
In our case the space $X$ will be of complex dimension $4n$.  The important conclusion will be that the isomonodromic flows span a rank-$2n$ distribution of the form (\ref{Lax_distribution}),  and thus,  as explained in \S \ref{paraconformal_from_lax} distinguish a family of holomorphic metrics on $X$.  \par
\subsection{Deformed polynomial oscillator potential} Take an integer $n \ge 1$.  We take the first term in the potential (\ref{potential})
\begin{align*}
Q_0(x) &:= x^{2n+1} + a_{n}x^{2n-1} + ... + a_1x^{n} + b_nx^{n-1} + ...  + b_1 
\end{align*}
to correspond to a meromorphic quadratic differential $Q_0(x) dx^2$ on $\mathbb{CP}^1$ with a single pole of order $2n + 5$ at $x = \infty$.  Next define 
\begin{align*}
Q_1(x) &:= \sum_{i=1}^n \frac{p_i}{x-q_i}+ R(x) 
\end{align*}
where $R(x)$ is the general polynomial of degree at most $n-1$.  The parameters $q_i$ specify the locations of $n$ singularities with residues $p_i^2 := Q_0(q_i)$.  We then define the trailing term $Q_2(x)$ as the unique meromorphic function with the condition that the Laurent expansion of $Q(x)$ at $x = q_i$ is of the form
\begin{equation}\label{apparent}
\frac{3}{4(x-q_i)^2} + \frac{u_i}{x-q_i} + u_i^2 + O(x-q_i)
\end{equation}
for some $u_i$ for each $i = 1,...,n$.   The condition ensures that the ODE (\ref{schro}) has `square-root-like' monodromy at the singularity $q_i$ (analytic continuation once around the singularity corresponds to multiplication by $1$ or $-1$).  Such a singularity is referred to as an \textit{apparent singularity} in the literature \cite{M}.   The only monodromy data that the parameters can affect is then at the irregular singularity at $x = \infty$.  \par 
In order to simplify calculation of the isomonodromic flows it is helpful to introduce parameters $(v_1,...,v_n)$ defined so that
$$
R(x) = \sum_{i=1}^n \Bigg(2p_iv_i - \sum_{k \ne i}^n \frac{p_k}{q_i-q_k}\Bigg)\prod_{j \ne i}^n \frac{(x-q_j)}{(q_i-q_j)}.  
$$
If $n = 1$ we have $v = r/2p$.  Explicitly 
\begin{equation*}
Q_2(x) = \sum_{i=1}^n \frac{3}{4(x-q_i)^2} + \sum_{i=1}^n \frac{v_i}{x-q_i} + S(x)
\end{equation*}
where
\begin{equation*}
S(x) = \sum_{i=1}^n \Bigg(v_i^2 - \sum_{k \ne i}^n\frac{3}{4(q_i-q_k)^2} - \sum_{k \ne i}^n \frac{v_k}{q_i-q_k}\Bigg) \prod_{j \ne i}^n \frac{(x-q_j)}{(q_i-q_j)}.  
\end{equation*}
Locally,  the equation (\ref{schro}) depends on $4n$ parameters $(a,b,q,v)$ where $a := (a_1,...,a_{n})$,  $b = (b_1,...,b_n)$ specify a quadratic differential $Q_0(x) dx^2$ and together with $q := (q_1,...,q_n)$,  $v := (v_1,...,v_n)$ specify the terms of lower degree in $\lambda$.  Globally there is some sign ambiguity in $p = (p_1,...,p_n)$ so we should think of the potential as corresponding to a point in a $4n$-complex-dimensional manifold $X$ with a fibration $\pi: X \to M$ given in local coordinates by $(a,b,q,v) \mapsto (a,b)$ with half dimensional fibres isomorphic to $\Sigma^0_{(a,b)} \times \mathbb{C}$ where 
 $\Sigma^0_{(a,b)} \subset \mathbb{C}^2$ is the affine part of the curve defined by $y^2 = Q_0(x)$.  We also require that $Q_0(x)$ has simple zeroes and so we take $M \subseteq \mathbb{C}^{2n}$ to be the open region on which this holds.  If $n = 1$ this is just the condition that $4a^3 + 27b^2 \ne 0$.   
 The $2n$-dimensional complex manifold $M$ is identified as the unfolding of the $A_{2n}$ singularity minus the discriminant locus \cite{D3}.  

\subsection{Calculation of the isomonodromic flows} 
We will now construct the isomonodromic flows in a similar fashion to the calculation in \cite{BM}.   A flow $(a(t),b(t),q(t),v(t))$ of the parameters in equation (\ref{schro}) depending on an auxiliary parameter $t$ is isomonodromic if there exists a matrix $B(x,t)$ of meromorphic functions such that the connection
\begin{equation*}
\nabla = d - \begin{bmatrix} 0 & 1 \\ Q(x,t) & 0 \end{bmatrix} dx - B(x,t)dt 
\end{equation*}
is flat (this condition is derived in a more general setting in \cite{U},  \cite{JMU}).    The general form of  $B(x,t)$ such that $\nabla$ is flat is
\begin{equation*}
B(x,t) = \begin{bmatrix} -\frac{1}{2}A' & A \\ AQ - \frac{1}{2}A'' & \frac{1}{2}A' \end{bmatrix} 
\end{equation*}
for some function $A(x,t)$,  
where an apostrophe denotes a derivative with respect to $x$,  and $A$ must satisfy
\begin{equation}\label{fuchs}
-2\frac{\partial Q}{\partial t} = \frac{\partial^3 A}{\partial x^3} - 4Q\frac{\partial A}{\partial x} - 2\frac{\partial Q}{\partial x}A.
\end{equation}
Substituting $Q$ and the Laurent expansion for $A$ at $x = q_i$ into (\ref{fuchs}),  a short calculation yields that $A$ has at most a simple pole at $q_i$.  
Next,  taking the Laurent expansion at $x = \infty$ of both sides of $(\ref{fuchs})$,  we see that $A$ is constrained to take the form
\begin{equation}\label{Ansatz}
A = \sum_{i=1}^n \frac{\zeta_i}{x-q_i}
\end{equation}
where the $\zeta_i$ are functions on $X$.  \par 
\begin{prop}[$A_{2n}$ isomonodromic flows]\label{A_2n flows prop}
There are $2n$ linearly independent isomonodromic flows for the equation {\rm(\ref{schro})}.  They are given by,  for $i = 1,...,n$,  
\begin{align*}
U_{i} &:= U_{i0'} +  \lambda{U_{i1'}} \nonumber \\ 
V_{i} &:= V_{i0'} + \lambda{V_{i1'}} 
\end{align*}
where
\begin{align}
\label{trivial_flow1}
U_{i0'} &:= -\sum_{j=1}^n \frac{q^{i-1}_j}{ 2p_j} \frac{\partial}{\partial v_j} \\
U_{i1'} &:= \frac{\partial}{\partial b_i} \label{trivial_flow2} 
\end{align}
and
\begin{align}
\label{bad_flow2}
V_{i0'} &:= -2p_i\frac{\partial}{\partial q_i} -  \sum_{j \ne i}^n \Bigg( \frac{\sum_{k={1}}^{n}T_k(q_i) q_j^{k+n-1}}{2p_j} + \frac{Q_0'(q_j)}{2p_j(q_i -q_j)} + \frac{p_j}{(q_i-q_j)^2} \Bigg)\frac{\partial}{\partial v_j}  \\ & \hspace{6mm}- \Bigg( \frac{\sum_{k={1}}^{n}T_k(q_i) q_i^{k+n-1}}{2p_i} - \frac{Q_0'(q_i)v_i}{p_i} + R'(q_i) - \sum_{j \ne i}^n \frac{p_j}{(q_i-q_j)^2}\Bigg)\frac{\partial}{\partial v_i}   \nonumber
\end{align}
\begin{align}
 \label{bad_flow1}
V_{i1'} &:= \sum_{j=1}^{n} T_j(q_i) \frac{\partial}{\partial a_j} - 2v_i\frac{\partial}{\partial q_i} + \sum_{j \ne i}^n \frac{1}{q_i-q_j}\frac{\partial}{\partial q_j} \\  \nonumber & \hspace{6mm}  - \sum_{j \ne i}^n \Bigg(\frac{3}{2(q_i-q_j)^3} + \frac{v_j}{(q_i-q_j)^2}\Bigg)\frac{\partial}{\partial v_j}  \\ & \hspace{6mm} + \Bigg(\sum_{j \ne i}^n \frac{3}{2(q_i-q_j)^3} + \sum_{j \ne i}^n\frac{v_j}{(q_i-q_j)^2} - S'(q_i)\Bigg) \frac{\partial}{\partial v_i} \nonumber 
\end{align}
where 
\begin{equation}\label{horizontal_flows}
T_j(x) := (2j-1)x^{n-j} + \sum_{k=1}^{n-j-1} (2j+k)a_{n-k+1} \, x^{n-j-1-k}.  
\end{equation}
\end{prop} 
\begin{proof} First we find $n$ flows $U_i$ that satisfy (\ref{fuchs}) with $A = 0$ and so preserve the potential in the sense $Q(x)$ is constant along the flows.  We call these isopotential flows.  Clearly $q_i$ must be constant along the flow since differentiation by $q_i$ leads to cubic terms which cannot be cancelled.  Similarly $a_i$ must be constant since there is nothing to cancel the $x^i$ term resulting from differentiation by this parameter.  If we think of the flows as being defined by a vector field $U$ on $X$ then $U$ must annihilate the terms in the Laurent expansion at $x = q_i$.  By (\ref{apparent}) a necessary condition is 
\begin{equation}
U(u_i) = 0, 
\end{equation}
for $i = 1,...,n$.  Conversely,  if this is the case then $U(Q(x))$ has no poles and vanishes at $q_1,...,q_n$,  but since $a_i$ does not flow we see that $U(Q(x))$ is a polynomial of degree at most $n-1$ and therefore must vanish since it has $n$ roots.  So it is in fact sufficient for $U$ to annihilate the $u_i$ in order to annihilate $Q(x)$.  As a corollary we see that there are $n$ linearly independent isopotential flows since we are searching for flows in $2n$ variables $b_1, ..., b_n,v_1,...,v_n$ which annihilate $n$ independent complex functions $u_i$ which we may calculate as 
\begin{equation}\label{udef}
u_i = \frac{p_i}{\lambda} + v_i. 
\end{equation}
In particular,  for these flows $\dot{p_i} = -\lambda \dot{v_i}$,  where we are denoting the derivative along the flow parameter $t$ by a dot.  Using the relation
\begin{align}\label{pflow}
2p_i\dot{p_i} = \dot{Q_0}(q_i) + Q_0'(\dot{q}_i) 
\end{align}
we see that for these flows $2p_i\dot{p}_i = \dot{b}_nq_i^{n-1} + ... + \dot{b}_1$ and after substituting to eliminate $\dot{p_i}$ we deduce the linearly independent flows are given by (\ref{trivial_flow1},  \ref{trivial_flow2})  above.  \par

Next we will find the flows which solve (\ref{fuchs}) with $A = (x-q_i)^{-1}$ for each $i$.  By linearity,  and the constraint (\ref{Ansatz}) this will give all the isomonodromic flows.  For the ensuing calculation we need the Laurent expansion of $Q(x)$ at $x = q_i$ to the next order.  For this write:
$$
Q'(x) = P_i(x) + O(({x-q_i})^{-1}),  \quad i = 1,...,n
$$
where $P_i(x)$ is holomorphic at $x = q_i$ and the omitted terms comprise the principal part of the Laurent expansion of $Q'(x)$ at $q_i$.  Then 
$$
Q(x) = \frac{3}{4(x-q_i)^2} + \frac{u_i}{x-q_i} + u_i^2 + P_i(q_i)\cdot(x-q_i)+ O\big((x-q_i)^2\big).  
$$
We will need the explicit expression for $P_i(x)$ later but it is best to leave it suppressed for the moment.  \par 
We now take (\ref{fuchs}) with $A = \frac{1}{x-q_i}$ and insist that the $(x-q_k)^{l}$ terms,  in the Laurent expansion of both sides of (\ref{fuchs}),  $l \le 0$,  agree for $k = 1,...,n$.  We need to handle the $k = i$ and $k \ne i$ terms separately.  For $k = i$ use the fact that the constant term in the Laurent expansion of
$$
4\frac{F(x)}{(x-q_i)^2} - 2\frac{F'(x)}{(x-q_i)}
$$ vanishes at $x = q_i$ for meromorphic $F(x)$.   The conditions are
\begin{equation*}
\begin{aligned}
(x-q_i)^{-3}:& \\
(x-q_i)^{-2}:& \\
(x-q_i)^{-1}:& \\
(x-q_i)^{0}\ \, :& 
\end{aligned}
\ \ \ \ \ 
\begin{aligned}
\dot{q}_i &= -2u_i \\
-2\dot{q}_i u_i &= 4u_i^2 \\ 
-2\dot{u}_i &= 2P_i(q_i) \\
-4u_i\dot{u}_i +  2\dot{q}_iP_i(q_i) &= 0.
\end{aligned}
\end{equation*}
A priori we have too many equations to specify $\dot{q_i}$ and $\dot{u_i}$ but the $(x-q_i)^0$ and $(x-q_i)^{-2}$ equations are implied by the other two.  For $k \ne i$,  a similar,  although somewhat more complicated computation taking the product of Laurent expansions gives
\begin{equation*}
\begin{aligned}
(x-q_k)^{-3}:& \\
(x-q_k)^{-2}:& \\
(x-q_k)^{-1}:& \\
(x-q_k)^{0}\ \, :& 
\end{aligned}
\ \ \ \ \ 
\begin{aligned}
\dot{q_k} &= -(q_k-q_i)^{-1} \\
-2\dot{q}_k u_k &= 2u_k(q_k-q_i)^{-1} \\ 
-2\dot{u}_k &= 2u_k(q_k-q_i)^{-2} - 3(q_k-q_i)^{-3} \\
-4u_k\dot{u}_k +  2\dot{q_k}P_k(q_k) &= -\frac{2P_k(q_k)}{q_k-q_i} -\frac{6u_k}{(q_k-q_i)^3} + \frac{4u_k^2}{(q_k-q_i)^{2}}
\end{aligned}
\end{equation*}
and similarly the $(x-q_k)^0$ and $(x-q_k)^{-2}$ equations are implied by the other two.  \par 
Given all the above are satisfied, (\ref{fuchs}) is a polynomial in $x$ with roots $x = q_k$,  $k=1,...,n$. The conditions give equations for $\dot{q_k}$ and $\dot{u_k}$ for $k = 1,...,n$.  Now considering (\ref{udef})
and (\ref{pflow}) we can solve for $\dot{v}_k$ given we know the $\dot{a}_k$ and $\dot{b}_k$ for $k = 1,...,n$.  Since the isopotential flows preserve the potential we can insist $\dot{b}_1,...,\dot{b}_n$ vanish.  Lastly,  $\dot{a}_{1},...,\dot{a}_{n}$ are determined by the condition that the $x^{2n-1},...,x^{n}$ terms in the Laurent expansion at $x = \infty$ of both sides of (\ref{fuchs}) agree.  It would then follow that (\ref{fuchs}) is a polynomial of degree at most $n-1$ with $n$ roots,  and therefore vanishes.  The relevant terms of (\ref{fuchs}) are
\begin{align}\label{highest_order}
-2\frac{\partial Q_0}{\partial t} = - 4Q_0\frac{\partial A}{\partial x} - 2\frac{\partial Q_0}{\partial x}A + O(x^{n-1})
\end{align}
where the trailing terms are of degree ${n-1}$ or lower in the expansion,  so can be ignored for this part of the calculation.  Equating terms in the Laurent expansion yields
$$
-2\dot{a}_{j} =  -2T_j(q_i)
$$
for $j = 1,...,n$.  \par Writing everything out in full using 
\begin{align*}
P_i(x) = \frac{Q_0'(x)}{\lambda^2} + \lambda^{-1}\Bigg(R'(x) - \sum_{j\ne i}^n\frac{p_j}{(x-q_j)^2}\Bigg) - \sum_{j\ne i}^n\frac{3}{2(x-q_j)^3} - \sum_{j\ne i}^n\frac{v_j}{(x-q_j)^2} +S'(x)
\end{align*}
 gives (\ref{bad_flow2},  \ref{bad_flow1}). 
\end{proof}
\begin{rem}\rm
Frobenius integrability of the span of the isomonodromic flows follows from general theory: The map which sends the parameters $(a,b,q,v)$ to the Stokes' data of the equation (\ref{schro}) is holomorphic,  and the isomonodromy distribution is the kernel of the differential.  Alternatively,  this can be shown by calculation from the equation (\ref{fuchs}).  
\end{rem}
\subsection{Quasi-homogeneity properties}
In what follows we will consider the vector field
\begin{equation}\label{general_euler}
W = \sum_{i=1}^{n} \bigg( \frac{(2n-2i+4)a_i}{2n+3}\frac{\partial}{\partial a_i} + \frac{(4n-2i+4)b_i}{2n+3}\frac{\partial}{\partial b_i} + \frac{2q_i}{2n+3}\frac{\partial}{\partial q_i} - \frac{2v_i}{2n+3}\frac{\partial}{\partial v_i} \bigg).  
\end{equation}
The explanation for the various weights in the rescaling of coordinates defined by $W$ is the following: these are the precise weights under which, together with rescaling $x$ with weight $2/(2n+3)$ and $\lambda$ with weight $1$ the Schr\"odinger equation (\ref{schro}) is preserved.  \par 
It is useful now to write down various quasi-homogeneity properties with respect to the rescaling defined by $W$.  From 
\begin{equation*}
W(p_j) = \frac{2n+1}{2n+3}\, p_j 
\end{equation*}
we deduce
\begin{align}\label{Vhomo}
[W,U_{i0'}] &= \, \ \ \frac{2i-2n-1}{2n+3}\, U_{i0'} \\
[W,U_{i1'}] &= -\frac{4n+4-2i}{2n+3}\,U_{i1'} 
\end{align}
for $i = 1,...,n$.  From 
\begin{align*}
W(T_j(q_i)) &= \frac{4n-2j}{2n+3}\, T_j(q_i) 
\end{align*}
where $j = 1,...,n$ we obtain 
\begin{align}\label{Uhomo}
[W,V_{i0'}] &= \frac{2n-1}{2n+3}\,V_{i0'} \\
[W,V_{i1'}] &= -\frac{4}{2n+3}\,V_{i1'}
\end{align}
for $i = 1,...,n$. 
\section{The $A_{2n}$ hyper-K\"ahler metric}\label{A_N metric}
In this section we will consider the family of metrics associated to the isomonodromic flows of Proposition \ref{A_2n flows prop} and identify a hyper-K\"ahler representative.  We will reinstate the summation convention with lower case primed indices taking values in $0',1'$ as before,  but with lower case unprimed indices taking values in  $1,...,n$.   \par  Consider the basis $\{U_i,V_i\}_{i=1}^n$ of isomonodromic flows from (\ref{A_2n flows prop}).  We will see the vector fields $\{U_{i0'},U_{i1'},V_{i0'},V_{i1'}\}_{i=1}^n$ trivialise $TX$ on the dense open subset of $X$ where the $q_i$ are distinct and the $p_i$ are non-zero.  On this open subset the flows define a Lax distribution of the form (\ref{Lax_distribution}),  taking $U_{ij'} = E_{ij'}$ and $V_{ij'} = E_{(i+n)j'}$ for $i = 1,...n$.   Given the dual basis of one-forms $\{U^{i0'},U^{i1'},V^{i0'},V^{i1'}\}_{i=1}^n$,  compatible metrics (\ref{Lax_metric}) for the associated $(2n,2)$-paraconformal structure are given by
\begin{equation}\label{paraconformal_class}
g =   2\epsilon_{i'j'}\big(\mu_{kl} U^{ki'} \odot U^{lj'} + 2\nu_{kl} U^{ki'} \odot V^{lj'} + \xi_{kl} V^{ki'}\odot V^{lj'}\big)
\end{equation}
specified by matrices $\mu_{kj},\nu_{kj},\xi_{kj}$ for which $\mu_{kj},\xi_{kj}$ are skew and in the notation of (\ref{Lax_metric})
\begin{align*}
e = 2\begin{bmatrix} \mu & \nu \\ -\nu\ & \xi \\ \end{bmatrix}. 
\end{align*}
In this trivialisation the quaternionic structures (\ref{quaternions}) are determined by 
\begin{align}\label{A2n_quaternions}
J(U_{j0'}) &= -U_{j1'} \hspace{5mm} K(U_{j0'}) = iU_{j1'}  \\
J(V_{j0'}\, ) &= -V_{j1'} \hspace{5mm}  K(V_{j0'}\,) = iV_{j1'} \nonumber
\end{align}
and the associated $2$-forms are given by
\begin{align}\label{fibre_skew_form}
\Omega_{+} = \mu_{kl} U^{k0'} \wedge U^{l0'} + 2\nu_{kl} U^{k0'} \wedge V^{l0'} + \xi_{kl} V^{k0'} \wedge V^{l0'}
\end{align}
and
\begin{align}\label{base_form}
\Omega_{-} = \mu_{kl} U^{k1'} \wedge U^{l1'} + 2\nu_{kl} U^{k1'} \wedge V^{l1'} + \xi_{kl} V^{k1'} \wedge V^{l1'}.  
\end{align}
Since the span of the isomonodromic flows is integrable for each value of $\lambda$,  the hyper-K\"ahler condition for $g$,  that $J,K$ (and hence $I$) are parallel for the Levi-Civita,  is equivalent to the closure of $\Omega_J,\Omega_K$.  \par 

Now $\Omega_J = \Omega_+ \, +\, \Omega_-$ and $\Omega_K = \Omega_+ \, -\, \Omega_-$.  
This means to find a hyper-K\"ahler metric,  we seek $\mu_{kj},\nu_{kj},\xi_{kj}$ for which $\Omega_+,\Omega_-$ are closed.  It will turn out the choice is quite natural to the geometry at hand.  \par 
As we will see,  the base $M$ carries a canonical symplectic form $\omega$ induced by the cohomology intersection pairing on each hyper-elliptic curve defined by $y^2 = Q_0(x)$.  \par 
In the next few subsections we will show the following:
 \begin{prop}[$A_{2n}$ hyper-K\"ahler metric]\label{distinguished_metric_prop}
The unique metric $g^\omega$ in the class \rm(\ref{paraconformal_class})\it \! such that the associated $2$-form $\Omega_{+}$ given by \rm(\ref{fibre_skew_form})\it \! satisfies 
\begin{equation}\label{fibre_darboux}
\Omega_+|_{ \operatorname{ker} \pi} = dv_i \wedge dq_i
\end{equation}
is equal to the unique metric in the class \rm (\ref{paraconformal_class})\it \! for which the associated $2$-form $\Omega_{-}$ given by \rm(\ref{base_form})\it \! satisfies $\Omega_{-} = \pi^*\omega$ and this metric is complex hyper-K\"ahler.  
 \end{prop}
 \begin{rem}\rm
 In the case $n=1$,  this metric agrees (up to a constant factor) with (\ref{PI1}).  
 \end{rem}
 \begin{rem}
 \rm
 In the langauge of \cite{BS},  the hyper-K\"ahler structure on $X$ defines an affine symplectic fibration for the symplectic form $\omega$. 
 \end{rem}
\subsection{Hyper-elliptic curves and the Gauss-Manin connection}
We first explain the definition of $\omega$.  A point $(a,b) \in M$ specifies non-singular genus-$n$ hyper-elliptic curve $\Sigma_{(a,b)}$ defined by the equation $y^2 = Q_0(x)$.  Take the holomorphic vector bundle $E \to M$ with fibres 
\begin{equation}
E_{{(a,b)}} := H^1(\Sigma_{(a,b)},  \mathbb{C}). 
\end{equation}
$E_{(a,b)}$ carries the usual intersection form $\langle \cdot , \cdot \rangle: H^1(\Sigma_{(a,b)},  \mathbb{C}) \times  H^1(\Sigma_{(a,b)},  \mathbb{C}) \to \mathbb{C}$ given by
\begin{equation}
\langle [\alpha] , [\beta] \rangle = \frac{1}{2}\int_{\Sigma_{(a,b)}} \alpha \wedge \beta
\end{equation}
for smooth representative one-forms $\alpha, \beta$.  Letting $(a,b)$ vary we get a holomorphic skew-form on $E$ we will also denote by $\langle \cdot , \cdot \rangle$.  \par 
Choose a basis $\{\mu_1,...\mu_{2n}\}$ dual to fundamental cycles $\{\gamma_1,...,\gamma_{2n}\}$.  The Gauss-Manin connection $\nabla^{GM}$ on $E$ is the flat connection constructed by decreeing sections taking values in $\{\mu_1,...\mu_{2n}\}$ at each point are covariantly constant \cite{CMP}.   \par 
Using $x$ as a local holomorphic coordinate on $\Sigma_{(a,b)}$ the differential form 
$$
Z_{(a,b)} = y dx 
$$
is clearly holomorphic away from the branch points but in fact extends to a meromorphic $1$-form on $\Sigma_{(a,b)}$ with a single pole at infinity.  To see this use $y$ as a holomorphic coordinate in a neighbourhood of a (finite) branch point and note 
$$
y dx =  \frac{2Q_0(x)}{Q_0'(x)} dy. 
$$
The branch point will be a root of $Q_0(x)$ but recalling the assumption that the roots of $Q_0(x)$ were distinct at points of $M$,  the denominator will not vanish.  A similar argument shows that $Z_{(a,b)}$ has no residue at the branch point $x = \infty$.  Since integrals of this one-form around cycles (away from infinity) are well defined $Z$ is a well defined section of $E$.  We will see that the map $\mu: TM \to E$ defined by 
$$
\mu(X) = \nabla^{GM}_X Z
$$
is an isomorphism of holomorphic vector bundles.  \par
It is a well-known fact that the one-forms $\eta_i$ with local representation
\begin{align}\label{holomorphic_1forms}
\eta_i =  \frac{\partial y}{\partial b_i}dx = \frac{x^{i-1}}{2y}dx, \  i = 1,...,n
\end{align}
extend to holomorphic $1$-forms on $\Sigma_{(a,b)}$,  as can be easily seen by looking at the orders of the numerator and denominator at $x = \infty$.  We will also need the one-forms $\omega_i$ with local representation
\begin{align*}
\omega_i = \frac{\partial y}{\partial a_i}dx = \frac{x^{n+i-1}}{2y}dx,  \ i = 1,...,n
\end{align*}
which are globally defined and meromorphic on $\Sigma_{(a,b)}$ with a pole of order $2i$ with zero residue at the branch point at $x = \infty$ \cite{B}.  \par It is important that $\eta_i$ and $\omega_i$ have no residues for $i = 1,...,n$ (they may traditionally be referred to as differentials of the first and second kind respectively).  This implies,  given a tangent vector $U \in TM$,  the integral of $U(y) dx$ around any trivial cycle vanishes and this meromorphic $1$-form $U(y) dx$ represents a cohomology class.  Furthermore,  there exists a meromorphic function $\phi_U$ defined in a neighbourhood of infinity such that $$ d\phi_U = U(y) dx.$$ \par 

Pulling $\langle \cdot,  \cdot \rangle$ back by the map $\mu$ and normalising by a constant factor for convenience we obtain a $2$-form 
$$\omega  := \frac{1 }{2\pi i}\mu^* \langle \cdot,  \cdot \rangle$$ 
on $M$.  
\begin{prop}[Hyper-elliptic intersection form] The $2$-form
$\omega$ is symplectic and given by the following formula
\begin{align}\label{intersection_residue}
\omega(U,V) = \frac{1}{2\pi i}\int_C U(y) \phi_V \,  dx. 
\end{align}
where $C$ is a contour in $\mathbb{CP}^1$ with the standard orientation,  containing all the poles of the integrand except one at $x = \infty$.  
\end{prop}
\begin{proof}
We have
\begin{align}\label{Phida_i}
\mu(U) = \sum_{L=1}^{2n} U\Big( \int_{\gamma_L} y dx \Big) \ \mu_L = \sum_{L=1}^{2n} \Big( \int_{\gamma_L} U(y)dx \Big) \ \mu_L.  
\end{align}
Take a smooth function $\rho$ which takes the constant value $1$ on a neighbourhood $D$ of $x = \infty$ and vanishes on the complement of a neighbourhood $
\tilde{D}$ of $D$.  
We may represent the cohomology class of $U(y)dx$  by a smooth,  globally-defined $1$-form as follows:
\begin{equation*}
\mu(U) = \big[U(y)dx- d(\rho \phi_U)\big]
\end{equation*}
and similarly for $V$.  We therefore have 
\begin{align*}
\omega(U,V) &= \frac{1}{4\pi i}\int_{\tilde{D}} d(\rho \phi_U) \wedge d(\rho \phi_V) - U(y) dx \wedge d(\rho \phi_V)  -  d(\rho \phi_U) \wedge V(y)dx
 \\ &= \frac{1}{4\pi i}\int_{\tilde{D}} d(2\rho \phi_V U(y)dx -\rho\phi_V d(\rho\phi_U)) \nonumber
\end{align*}
where we have used $d(\rho \phi_V U(y)dx) =  - d(\rho\phi_U V(y)dx)$.  
Next we use Stokes' theorem,  taking the boundary contour $\partial D$ so that we can set $\rho = 1$ in the integrand.  
The $1$-form $U(y)\phi_Vdx$ being integrated is pulled-back from $\mathbb{CP}^1$ so we may write the integral as a contour integral in $\mathbb{CP}^1$.  We take the contour $C$ in $\mathbb{CP}^1$ with preimage $\partial D$ under the covering map to have a winding number of 1 anti-clockwise around the origin,  which means we need to multiply the integrand by a factor of $2$ as $\partial D$ traverses both sheets. 
\begin{align*}
\omega(U,V) &= \frac{1}{4\pi i}\oint_{\partial D} U(y) \phi_V dx = \frac{1}{2\pi i}\oint_{C} U(y) \phi_V dx,  
\end{align*} 
which is the desired formula.  \par 
To see that $\omega$ is non-degenerate it suffices to show that $\omega$ is represented by a triangular matrix with non-vanishing diagonal entries.  Use the fact the wedge product of holomorphic $1$-forms vanishes to see
\begin{align*}
\omega\bigg(\frac{\partial}{\partial b_j},\frac{\partial}{\partial b_k}\bigg) = 0.  
\end{align*}
Next by checking the order of the pole at which the residue is extracted we see that
\begin{align*}\label{triangular}
\omega \bigg(\frac{\partial }{\partial a_j},  \frac{\partial }{\partial b_k} \bigg) = \,  &\operatorname{res}_{x = \infty} \bigg( \frac{\partial y}{\partial a_j} \,  d^{-1} \bigg( \frac{\partial y}{\partial b_k}  dx \bigg) dx \bigg) 
\end{align*}
vanishes for $j + k \le n + 1$ and 
\begin{align*}
\omega \bigg(\frac{\partial }{\partial a_j},  \frac{\partial }{\partial b_{n+1-j}} \bigg)  = \operatorname{res}_{x = \infty} \bigg( \frac{\partial y}{\partial a_j} \,  d^{-1} \bigg( \frac{\partial y}{\partial b_{n+1-j}}  dx \bigg) dx \bigg) =  \frac{1}{2-4j} \ne 0.  
\end{align*}
The closure of $\omega$ follows from the Schwarz rule and integration by parts.  
\end{proof}
Since $\omega$ is non-degenerate it follows $\mu: TM \to E$ is an isomorphism.  \par 

For the first few values of $n$ we have
\begin{align*}
A_2: \hspace{5mm}  \omega = \ &\frac{1}{2}db \wedge da \\ 
A_4: \hspace{5mm} \omega = \ &\frac{1}{6}db_1 \wedge da_2 + \frac{1}{2}db_2 \wedge da_1 + \frac{a_2}{6}da_1 \wedge da_2 \phantom{\Bigg(} \\
A_6: \hspace{5mm} \omega = \ &\frac{1}{10}db_1 \wedge da_3 + \frac{1}{6}db_2 \wedge da_2 + \frac{1}{2} db_3 \wedge da_1 - \frac{3a_3}{10}db_3 \wedge da_3 \\ &+ \frac{a_3}{6}da_1 \wedge da_2 + \frac{a_2}{5}da_1 \wedge da_3 + \Big(\frac{a_3^2}{10} + \frac{a_1}{30}\Big)da_2 \wedge da_3 \nonumber. 
\end{align*}
  \subsection{Distinguished $A_{2n}$ metric}
It will be convenient to introduce matrices $A_i{}_{j},\tilde{B}_{i}{}_{j}, \hat{B}_{i}{}_{j}$,  $\tilde{C}_{i}{}_{j}, \hat{C}_{i}{}_{j},  T_{ij}$ to compactly write the flows from (\ref{A_2n flows prop}) as follows:
 \begin{align}
  U_{i0'} &= A_{i}{}_{j} \frac{\partial}{\partial v_j} \\
 U_{i1'} &= \frac{\partial}{\partial b_i} \nonumber \\
  V_{i0'} &= \tilde{C}_{i}{}_{j}\frac{\partial}{\partial q_j } + \hat{C}_{i}{}_{j}\frac{\partial}{\partial v_j} \nonumber \\
    V_{i1'} &= T_{ij}\frac{\partial}{\partial a_j} + \tilde{B}_{i}{}_{j}\frac{\partial}{\partial q_j } + \hat{B}_{i}{}_{j}\frac{\partial}{\partial v_j} \nonumber
 \end{align}
 for $i = 1,...,n$.  Where the $q_i$ are distinct and $p_i$ are non-zero,  the matrices $A_{ij}$,  $T_{ij}$ and $\tilde{C}_{ij}$ are invertible.  To see this note that the matrices $A_{ij}$ and $T_{ij}$ can each be written as the product of the Vandermonde matrix $q_j^{i-1}$ and a triangular matrix with non-zero diagonal entries,  while $\tilde{C}_{ij}$ is diagonal with the $p_i$ as its entries.  This implies that $\{U_{i0'},U_{i1'},V_{i0'},V_{i1'}\}_{i=1}^n$ trivialise $TX$ on this open region.  We then calculate the dual trivialisation as
 \begin{align}\label{dual_trivialisation}
  U^{i0'} &= A^{-1}_{ji}\,dv_j - A^{-1}_{ki}\hat{C}_{l}{}_{k}\tilde{C}{}^{-1}_{jl}\,dq_{j} - A^{-1}_{ki}\hat{B}_{l}{}_{k}T^{-1}_{jl} \,  da_j \\  
 & \hspace{5mm} + A^{-1}_{ki}\hat{C}_{l}{}_{k}\tilde{C}^{-1}_{ml}\tilde{B}_{r}{}_{m}T^{-1}_{jr} \,  da_j \nonumber \nonumber \\
 U^{i1'} &= db_i \nonumber \\ 
  V^{i0'} &=  \tilde{C}^{-1}_{ji} \,  dq_j - \tilde{C}^{-1}_{ki}\tilde{B}_{l}{}_{k}T^{-1}_{jl} \,  da_j \nonumber \\
 V^{i1'} &= T^{-1}_{ji} \,  da_j \nonumber 
 \end{align}
 where the components $(A^{-1})_{ij}$ of the matrix inverses are denoted by $A^{-1}_{ij}$ and satisfy $A_{ij}A^{-1}_{jk} = \delta_{ik}$ and so on.\par 
 Suppose that we would like to choose the the metric $g$ in the family (\ref{paraconformal_class}) so that the $2$-form $\Omega_{+}$ takes Darboux form when restricted to the vertical bundle as in Proposition\ref{distinguished_metric_prop}.  Specifically
 \begin{equation*}
\Omega_+|_{\operatorname{ker} \pi} = dv_i \wedge dq_i. 
\end{equation*}
Substituting the expressions for the dual trivialisation (\ref{dual_trivialisation}) into (\ref{fibre_skew_form}) we see that $\mu_{ij} = 0$ since otherwise $\Omega_+|_{\operatorname{ker} \pi}$ would contain terms of the form $\mu_{ij}A^{-1}_{ki}A^{-1}_{lj}\,dv_{k} \wedge dv_{l}$.  Given this simplification 
\begin{align*}
\Omega_+|_{\operatorname{ker} \pi}  &= 2\nu_{ij}A^{-1}_{ki}\tilde{C}^{-1}_{lj} \,  dv_{k} \wedge dq_{l}+ \big(\xi_{ij}\tilde{C}^{-1}_{mi}\tilde{C}^{-1}_{rj}-2\nu_{ij}A^{-1}_{ki} \hat{C}_{l}{}_{k}\tilde{C}^{-1}_{ml}\tilde{C}^{-1}_{rj} \big) \,  dq_m \wedge dq_r.   
\end{align*}
So we see the following:
\begin{lemma}\label{metric_explicit}
There is a unique metric $g^{\omega}$  in the class \rm(\ref{paraconformal_class})\it \! such that the associated $2$-form $\Omega_{+}$ given by \rm(\ref{fibre_skew_form})\it \! satisfies \rm(\ref{fibre_darboux}).\it \!  It is obtained by setting
\begin{align}\label{matrix_components}
\mu_{ij} &= 0 \\ 
\nu_{ij} &= \frac{1}{2}A_{i}{}_{k}\tilde{C}_j{}_{k} \nonumber \\ 
\xi_{ij} &= \frac{1}{2}\big(\hat{C}_{i}{}_{k}\tilde{C}_{j}{}_{k} -\hat{C}_{j}{}_{k}\tilde{C}_{i}{}_{k}\big). \nonumber
\end{align}
\end{lemma}
A calculation using the explicit form of $A_i{}_j,  \tilde{C}_i{}_j,  \hat{C}_i{}_j$ then gives give the non-zero components of $\Omega_{-}$ with respect to the trivialisation $\{U_{i0'},U_{i1'},V_{i0'},V_{i1'}\}_{i=1}^n$: 
\begin{align*}
\Omega_{-}(U_{j1'},U_{k1'}) &= 0  \phantom{\frac{1}{2}} \\ 
\Omega_{-}(U_{j1'},V_{k1'}) &= -\frac{A_{j}{}_{l}\tilde{C}_k{}_{l}}{2} =  \frac{q_k^{j-1}}{2}\\
\Omega_{-}(V_{j1'},V_{k1'}) &=\frac{1}{2}\big(\tilde{C}_{j}{}_{l}\hat{C}_{k}{}_{l} -\tilde{C}_{k}{}_{l}\hat{C}_{k}{}_{l}\big) = \frac{1}{2}\big({F_j(q_k) - F_k(q_j)}\big)
\end{align*}
where
\begin{align}\label{Fdef}
F_i(x) := V_{i1'}(Q_0(x)) - \frac{Q_0'(x)}{x-q_i} + \frac{2Q_0(x)}{(x-q_i)^2}.  
\end{align}
 \begin{lemma}\label{omega_-_asf}
  The $2$-form $\, \Omega_{-}$ given by \rm (\ref{base_form}) \it associated to $g^\omega$ satisfies $\Omega_- = \pi^*\omega$.  
 \end{lemma}
 \begin{proof}
First rewrite (\ref{Fdef}) in terms of $1$-forms on $\Sigma_{(a,b)}$ as follows:
 \begin{align*}
\frac{ F_j(x)}{2y} dx = V_{j1'}(y) dx  -  \frac{Q_0'(x)}{2y(x-q_j)} dx + \frac{2Q_0(x)}{2y(x-q_j)^2} dx.  
 \end{align*}
The first term on the right hand side has a meromorphic antiderivative on a neighbourhood of $x = \infty$.   The last two terms together are a total derivative of $y/(x-q_j)$.  Write $I_j(x)$ for the (local) antiderivative of the left hand side.  \par
Applying the formula (\ref{intersection_residue}) then yields
 \begin{equation}\label{omega_pullback}
  \pi^*\omega(V_{j1'},V_{k1'}) = \frac{1}{2\pi i}\oint_{C} \bigg(\frac{F_j(x)}{2y} + \frac{\frac{\partial y}{\partial x}}{(x-q_j)} - \frac{y}{(x-q_j)^2}\bigg)\cdot\bigg( I_k(x) + \frac{y}{(x-q_k)} \bigg)\ dx.  
 \end{equation}
 Let $\tilde{x} = 1/x$.  
In the proof of (\ref{A_2n flows prop}) the vector field $V_{i1'}$ was chosen precisely so that the Laurent expansion at infinity of the right hand side of (\ref{Fdef}) has a pole of order at most $n-1$ at infinity (see (\ref{highest_order})).   That is 
 \begin{equation*}
 F_j(x) = c\tilde{x}^{-n+1} + O(\tilde{x}^{-n+2})
 \end{equation*}
 for some constant $c$.  
The contour integral (\ref{omega_pullback}) is proportional to the coefficient of the $\tilde{x}$ term in the Laurent expansion of the integrand at infinity.  
 Meanwhile we have the Puiseux series
\begin{equation*}
y^{-1} = \tilde{x}^{n+\frac{1}{2}}- \frac{1}{2}a_{n}\tilde{x}^{{n+2}} + O(\tilde{x}^{{n+3}}).  
\end{equation*} 
Therefore,  the $1$-form $
F_j(x)I_k(x)/2y \, dx$
does not contribute to the contour integral since it is holomorphic.  We may evaluate other terms 
by integration by parts which yields,  after skew-symmetrising 
\begin{equation*}
  \pi^*\omega(V_{j1'},V_{k1'}) = \frac{1}{4\pi i}\oint_{C} \frac{F_j(x)}{x-q_k} - \frac{F_k(x)}{x-q_j} + \frac{Q_0(x)}{(x-q_k)^2(x-q_j)}  - \frac{Q_0(x)}{(x-q_j)^2(x-q_k)} \ dx.
\end{equation*}
Evaluating this integral by taking residues at the poles $x = q_j,  q_k$ (note that $F_j(x)$ has a simple pole at $x = q_j$) then yields,  after some cancellation,  the desired equality$$  \pi^*\omega(V_{j1'},V_{k1'})  = \frac{1}{2}\big({F_j(q_k) - F_k(q_j)} \big).$$
Similarly,  noting that $x^{j-1}I_k(x)/2y \, dx$ is holomorphic gives
\begin{align*}
  \pi^*\omega(U_{j1'},V_{k1'}) = \frac{1}{2\pi i}\oint_{C}\frac{x^{j-1}}{2y} \cdot \bigg(I_k(x)+ \frac{y}{(x-q_k)}\bigg)dx = \frac{1}{2\pi i}\oint_C \frac{x^{j-1}}{2(x-q_k)} \ dx = \frac{q^{j-1}_k}{2}. 
\end{align*}
Lastly,  that the wedge product of holomorphic $1$-forms vanishes implies
$$ \pi^*\omega(U_{j1'},U_{k1'}) = 0.$$
Altogether we have shown $\pi^*\omega$ and $\Omega_{-}$ agree when evaluated on the elements of the trivialisation $\{U_{i0'},U_{i1'},V_{i0'},V_{i1'}\}_{i=1}^n$ which is sufficient to show the desired equality.  
\end{proof}
The closure of $\Omega_{-}$ follows from the closure of $\omega$.  The final thing to check to show $g^\omega$ is hyper-K\"ahler is therefore:
 \begin{lemma}\label{omega_+_closure}
 $\Omega_+$ associated to $g^\omega$ is closed.  
 \end{lemma}
 \begin{proof}
Let $U, V, W$ be tangent vectors.  The standard formula for the exterior derivative is
\begin{align}\label{exterior_derivative}
d\Omega_+(U,V,W) = \ &U(\Omega_+(V,W)) + V(\Omega_+(W,U)) +  W(\Omega_+(U,V)) \\   - \ &\Omega_+([U,V],W) - \Omega_+([V,W],U) - \Omega_+([W,U],V). \nonumber
\end{align}
Now $TX = \operatorname{span}\{U_{i0'},  V_{i0'}\} \oplus \operatorname{span}\{U_{i1'},  V_{i1'}\} $.  The first summand is the vertical bundle while the second summand is the kernel of $\Omega_+$.  Clearly (\ref{exterior_derivative}) vanishes when $U,V,W$ all belong to the second summand.  Next,  Frobenius integrability of the second summand implies (\ref{exterior_derivative}) vanishes when $U$ is vertical and $V,W$ belong to the second summand.
Substituting (\ref{dual_trivialisation}) together with (\ref{matrix_components}) into (\ref{fibre_skew_form}) we obtain (after some cancellation): 
\begin{align}\label{relevant_terms}
\Omega_{+} = dv_i \wedge dq_i - \tilde{B}_{l}{}_{i}T^{-1}_{kl} \,  dv_i \wedge da_k - \hat{B}_{k}{}_{i}T^{-1}_{lk} \,  da_l \wedge dq_i + ... 
\end{align}
where the omitted terms vanish when one input is vertical.  It follows that (\ref{exterior_derivative}) vanishes when $U, V, W$ are all vertical.   It remains to show (\ref{exterior_derivative}) vanishes when $U,V$ are vertical and $W$ belongs second summand.  For this it is sufficient to check that the exterior derivative of the terms written out in (\ref{relevant_terms}) vanish when two inputs are vertical.  Because the $T_{ij}$ are independent of the $v_i$ this reduces to 
\begin{align}
\frac{\partial\tilde{B}_{k}{}_{i}}{\partial v_{j}} - \frac{\partial\tilde{B}_{k}{}_{j}}{\partial v_{i}} &= 0 \label{comb1} \\
\frac{\partial}{\partial q_i}\Big(T^{-1}_{kl}\hat{B}_{l}{}_{j}\Big) - \frac{\partial}{\partial q_j}\Big(T^{-1}_{kl}\hat{B}_{l}{}_{i}\Big) &= 0 \label{comb2}  \\
\frac{\partial \hat{B}_{k}{}_{i}}{\partial v_j} + \frac{\partial\tilde{B}_{k}{}_{j} }{\partial q_i} + T_{kl}\frac{\partial \, T^{-1}_{lm}}{\partial q_i}\tilde{B}_{m}{}_{j} &= 0 \label{comb3}
\end{align}
for $k = 1,...,n$.  Checking (\ref{comb1}) is straightforward.  The other two equalities can be shown by long (although elementary) calculations.  It helps to use
\begin{align*}
T_{jk}\frac{\partial T^{-1}_{kl}}{\partial q_i} = -\frac{\partial T_{jk}}{\partial q_i}T^{-1}_{kl}
\end{align*}
and the fact we can decompose $T_{ij} = F_{i}{}_{k}Q_{k}{}_{j}$ where $F_{i}{}_{k}$ is independent of the $q_l$ and $Q_{k}{}_{j} := q_k^{j}$.  Then 
\begin{align}\label{T_matrix_identity}
 \frac{\partial T_{jk}}{\partial q_i}T^{-1}_{lk} = \frac{\partial Q_{j}{}_{k}}{\partial q_i}Q^{-1}_{kl}.  
 \end{align}
 The right hand side of (\ref{T_matrix_identity}) vanishes unless $i = j$ in which case 
 \begin{align*}
\sum_{k=1}^n \frac{\partial Q_{i}{}_{k}}{\partial q_i}Q^{-1}_{kl} &= \frac{\prod_{k\ne i,n}(q_i-q_k)}{\prod_{k\ne l}(q_n-q_k)},  \ l \ne i \\
 \sum_{k=1}^n   \frac{\partial Q_{i}{}_{k}}{\partial q_i}Q^{-1}_{ki} &= \sum_{k \ne i} \frac{1}{(q_i-q_k)}
 \end{align*}
 where we are suppressing the summation convention.  
 \end{proof}
 Together with Lemma \ref{metric_explicit} and Lemma \ref{omega_-_asf} this completes the proof of Proposition \ref{distinguished_metric_prop}.
 \begin{prop}\label{prop_homothety}
$\mathcal{L}_Wg^\omega = g^\omega $ where $W$ is given by \rm(\ref{general_euler}).  
 \end{prop}
\proof This is straightforward to verify using (\ref{Vhomo},  \ref{Uhomo}) and
\begin{align*}
W(A_{ij}) &= \frac{2i-2n-3}{2n+3}A_{ij} \\ 
W(\tilde{C}_{ij}) &= \frac{2n+1}{2n+3}\tilde{C}_{ij} \\
W(\hat{C}_{ij}) &= \frac{2n-3}{2n+3}\hat{C}_{ij}.  
\end{align*}
\begin{prop}
\label{prop_final}
The metric $g^\omega$ admits a projectable hyperlagranian foliation $\mathcal{U}$. 
\end{prop}
\begin{proof} Clearly $\mathcal{U} := \text{span}\{U_{i0'},U_{i1'}\}_{i=1,...,n}$ is an integrable distribution.  That $\mu_{kl} = 0$ ensures it is hyper-Lagrangian.  Its image under $d\pi$ is the Lagrangian foliation for $\omega$ given by the span of the $b_i$ coordinate fields.  
\end{proof}
Theorem  \ref{theo3_intro}
now follows from Proposition \ref{distinguished_metric_prop},  Proposition \ref{prop_homothety} and Proposition \ref{prop_final}.  \par 
\begin{rem}
\rm
The foliation $\mathcal{U}$ satisfies a stronger condition than being hyper-Lagrangian:  There is a $2n-2$-dimensional family of $\beta$-surfaces.  They are given by the leaves of $\text{span}\{\sigma^iU_{i0'},\sigma^iU_{i1'}\}$ where $\sigma^i$ are constants for $i = 1,...,n$. 
\end{rem}
\begin{rem}
\rm 
$\mathcal{U}$ has an interpretation in terms of the Hodge structure of the hyper-elliptic curves specified by points in $M$.    
The image of $d\pi(\mathcal{U})$ under the isomorphism $\mu: TM \to E$ is the rank-$n$ sub-bundle with fibre $H^{1,0}(\Sigma_{(a,b)},\mathbb{C})$.  
\end{rem}
\section{Outlook}
We have constructed a family of complex hyper--K\"ahler metrics associated with the isomonodromy problem for the deformed polynomial oscillator of degree $2n+1$.  These are compatible with the symplectic structure on a space $M$ of meromorphic quadratic differentials on $\mathbb{CP}^1$ with a single pole of order $2n+5$.  
If $n=1$, then this metric corresponds to the $A_2$ Joyce structure of \cite{BM}, and arises from isomonodromic flows defining Painlev\'e I.  One motivation for this
explicit construction is as follows: For general $n$, the base $M$ can be identified as the universal unfolding of the $A_{2n}$ singularity which has a standard Frobenius structure in a family studied by Saito \cite{S} and Dubrovin \cite{D3}.  In the $n=1$ case,  the Joyce structure is compatible with this Frobenius structure in the sense of \cite{B3}; a Joyce structure induces a bilinear form on the base $M$ which in this case agrees with the Frobenius metric up to a constant scale.  In this paper we did not go as far as to check this for general $n$ but the expressions obtained here would be the starting point for such a calculation.  
\par
We did not go as far to show that the metrics constructed in \S\ref{A_N metric} give rise to Joyce structures.  We expect the construction to not differ drastically from \cite{BM}.  After choosing a basis of cycles,  a map $X_{(a,b)} \to T_{(a,b)}M / H^1(\Sigma_{(a,b)},\mathbb{Z})$ can be constructed by considering integrals of the differential form  
$$
\varpi := \frac{Q_1(x)dx}{2y} + \sum_{i=1}^{n} \frac{dx}{2(x-q_i)}
$$
around the basis of cycles.  That the image is an open dense set follows due to Jacobi inversion.  The hyper-K\"ahler metric can be pushed forward by this map in each fibre since it is invariant under exchange of the $q_i$.  The Joyce structure should be obtained by pulling the structure back to $TM$,  and this leads to the lattice invariance property.   Finally,  Pleba\'nski coordinates $\zeta^i$ and $\theta^i$ for $X$ are given by the integrals of $y dx$ and $\varpi$ (defined up to integer multiples of $2\pi i$) respectively,  around the cycles.
\par
The heavenly potential corresponding to the metric of Theorem \ref{theo3_intro} will be quadratic in half of the fibre coordinates in a suitably adapted Pleba\'nski coordinate system.  This is a general property (Theorem \ref{theo1_intro}) of metrics admitting a projectable hyper--Lagrangian foliation. In dimension four this property is sufficient (Theorem \ref{theo2_intro})
to linearise the heavenly equation. It would be interesting to know whether this linearisation generalises
to the heavenly system (\ref{heavenly2}) in higher dimensions.  
\par 
In  \cite{alexandrov1, alexandrov2} Alexandrov and Pioline interpreted  the heavenly equation corresponding to Joyce structures in terms of the thermodynamics Bethe ansatz, and obtained integral representations
for the heavenly potential $\Theta$.  It would be interesting to put our $A_{2n}$  solution in this framework, and understand it as a conformal limit.
The connection may arise by the iterative construction of the holomorphic
coordinates on the twistor space associated to the $A_{2n}$ metric. Rather than
employing a recursion operator procedure of \cite{DM} to construct such twistor functions,
one may attempt to  iterate the integral equations of the TBA as in \cite{Alday} .

\par 
\section*{Appendix. Cohomogeneity one Painlev\'e I metric}
\appendix
\renewcommand{\theequation}{A.\arabic{equation}}
\setcounter{equation}{0}

In \S\ref{pflow1} we constructed a complex hyper--K\"ahler metric (\ref{PI1}) from a Lax pair $\{l_1,l_2\}$ such that the integral curves of $l_2$ were defined by the Painlev\'e I equation.  In dimension four all  complex hyper--K\"ahler metrics are conformally anti--self--dual, and null--K\"ahler in the sense of Definition \ref{defnN}.

In \cite{D} another anti--self--dual null--K\"ahler metric was constructed using the formulation of Painlev\'e I as a reduction of the isomonodromy equations and the twistor construction
of Hitchin \cite{H}. This led to a question of whether these
two metrics were the same, or at least conformally related.  In this section
we shall settle this 
and deduce that, despite being anti--self--dual
and null--K\"ahler, these two metrics are different.

\vskip5pt
Consider the isomonodromic Lax pair  with irregular singularity
of order four
\begin{eqnarray}
\label{eq1}
M_1&=&\frac{\p}{\p\lambda}  
-\left(\begin{array}{cc}
-z&y^2+\frac{t}{2}\\
-4y&z
\end{array}
\right)-
\lambda 
\left(\begin{array}{cc}
0&y\\
4&0
\end{array}
\right)
-\lambda^2\left(\begin{array}{cc}
0&1\\
0&0
\end{array}
\right)\\
M_2&=&\frac{\p}{\p t} -\left(\begin{array}{cc}
0&y\\
2&0
\end{array}
\right)-
\frac{1}{2}\lambda\left(\begin{array}{cc}
0&1\\
0&0
\end{array}
\right),\nonumber
\end{eqnarray}
where $y=y(t), z=z(t)$.
The compatibility condition $[M_1, M_2]=0$ is equivalent to
\be
\label{PI}
\dot{y}=z, \quad \dot{z}=6y^2+t\quad\mbox{so that}\quad \ddot{y}=6y^2+t
\ee
which is the Painlev\'e I equation. 
Now replace matrix generators of $\mathfrak{sl}(2, \C)$ by left-invariant vector fields
$(l_1, l_2, l_3)$ on $SL(2, \C)$ according to
\[
\left(\begin{array}{cc}
1/2&0\\
0&-1/2
\end{array}
\right)\rightarrow l_1,\quad 
\left(\begin{array}{cc}
0&1\\
0&0
\end{array}
\right)\rightarrow l_2,\quad
\left(\begin{array}{cc}
0&0\\
1&0
\end{array}
\right)\rightarrow l_3,
\]
and remove the quadratic term from $M_1$ by taking the linear combination
\[
L_1=2\lambda M_2-M_1, \quad L_2=2M_2.
\]
This results in a Frobenius integrable Lax distribution $\{L_1,L_2\}$ of the form (\ref{Lax_distribution}) on $\X\times \CP^1$, where the four--manifold
$\X$ is a product $SL(2, \C)\times \C$,  the four vector fields $E_{ij'}$ are given by
\begin{align}
E_{10'}=-2z l_1+\Big(y^2+\frac{t}{2}\Big)l_2-4yl_3  , \quad E_{21'}=l_2, \\
E_{11'}=-2\frac{\p}{\p t}+yl_2,   
\quad
E_{20'}=2\frac{\p}{\p t}-(2yl_2+4l_3), \nonumber
\end{align}
and $f_1=-1, f_2=0$.

To write the corresponding anti-self-dual metric explicitly in terms of the Painlev\'e  I transcendent introduce the left--invariant one--forms on $SL(2, \C)$ such that $\sigma_i(l_j) = \delta^i_j$, and
\[
d\sigma^1=2\sigma^3\wedge \sigma^2, \quad d\sigma^2=\sigma^2\wedge \sigma^1,\quad
d \sigma^3=\sigma^1\wedge\sigma^3.
\]
Then $g$, when rescaled by $16z(t)$ takes the cohomogeneity-one form
\be
\label{metricpi}
g=\sigma^1\odot\Big(\frac{12y^2+2t}{z}\sigma^1+8\sigma^2-6\sigma^3\Big)+\sigma^3\odot(z\sigma^3+2z dt).
\ee
The $2$-form $2\sigma^3\wedge\sigma^1$ is parallel as a consequence of Painlev\'e I. Consequently the metric is null-K\"ahler.

It gives rise to a two--parameter family of $\alpha$--surfaces in $\X$ spanned by vectors in its kernel.
This family gives a divisor line bundle in the complex three--fold ${\mathcal Z}$ (this is the twistor
space whose points are all $\alpha$--surfaces in $\X$). This is a canonical divisor - a section $s$
of the $1/4$ power of the holomorphic canonical bundle of ${\mathcal Z}$ - in fact the existence
of such section singles out a null--K\"ahler metric in an anti-self-dual conformal class \cite{D}.  The twistor space 
${\mathcal Z}$ arises as a quotient
\[
{\X}\stackrel{r}\longleftarrow 
{\mathcal F}\stackrel{q}\longrightarrow {\Y},
\]
of the correspondence space ${\mathcal F}=\C\times SL(2, \C)\times \CP^1$ by the distribution (\ref{Lax_distribution}). There is an $SL(2, \C)$ action on the 
twistor space, and the holomorphic vector fields generating this action
become linearly dependent on the divisor $s=0$. The divisor is preserved by
the action, so there is an element of the Lie algebra that maps to zero
under the associated vector bundle homomorphism 
\be
\label{nigel}
\phi:\mathfrak{sl}(2, \C)\times{\Y}\rightarrow T{\Y}.
\ee
This element is nilpotent, which distinguishes this
from the twistor space of Painlev\'e II (where the divisor $s=0$ also meets
each twistor line to order $4$).
The inverse of $\phi$ is the $SL(2, \C)$ connection with a pole of order 4 on the divisor 
$\{s=0\}$
underlying the isomonodromy problem for Painlev\'e I.

\vskip5pt
At this point we can compare the metric (\ref{metricpi}) to the ASD Ricci-flat metric (\ref{PI1}). Both metrics
are anti--self--dual, null--K\"ahler, and arise from the isomonodomy problem of Painlev\'e I but that is where the similarities end. The metric (\ref{metricpi}) has cohomogeneity one, so any conformally related metric
will admit $SL(2, \C)$ acting by conformal isometries. On the other hand the metric (\ref{PI1}) has only one conformal Killing vector. Therefore the two are neither locally isometric, nor conformally related.

\end{document}